\documentclass[a4paper, 12pt]{article}
\usepackage{xr}
\usepackage{amsmath}
\usepackage{amsthm}
\usepackage{graphicx}
\usepackage{amsfonts}
\usepackage[usenames]{color}
\usepackage{amssymb}
\usepackage{amsbsy}
\usepackage{graphics}
\usepackage{cite}
\usepackage[hidelinks]{hyperref}

\usepackage{bbold}
\usepackage[a4paper,bindingoffset=0.2in,%
            left=1in,right=1in,top=2in,bottom=1in,%
            footskip=.25in]{geometry}

\usepackage[rflt]{floatflt}
\usepackage{wrapfig} 
\usepackage{subfigure} 

\usepackage{scalerel,stackengine}
\stackMath
\newcommand\rwh[1]{%
\savestack{\tmpbox}{\stretchto{%
  \scaleto{%
    \scalerel*[\widthof{\ensuremath{#1}}]{\kern-.6pt\bigwedge\kern-.6pt}%
    {\rule[-\textheight/2]{1ex}{\textheight}}
  }{\textheight}%
}{0.5ex}}%
\stackon[1pt]{#1}{\tmpbox}%
}

\usepackage[affil-it]{authblk}

\newtheorem{theorem}{\bf Theorem}[section]

\newtheorem{corollary}{\bf Corollary}[section]
\newtheorem{definition}{\bf Definition}[section]
\newtheorem{remark}{\bf Remark}[section]
\newtheorem{proposition}{\bf Proposition}[section]
\newcommand\blfootnote[1]{%
  \begingroup
  \renewcommand\thefootnote{}\footnote{#1}%
  \addtocounter{footnote}{-1}%
  \endgroup
}

\begin{document}
\date{}
\vspace{-2.5cm}
\title{\vspace{-1cm} \bf Lie algebra structure of fitness and replicator control}
\author[1,2,3]{Vidya Raju$^{*}$}
\author[1,2]{P. S. Krishnaprasad$^\dagger$}
\affil[1]{Department of Electrical and Computer Engineering, University of Maryland, College Park}
\affil[2]{Institute for Systems Research, University of Maryland, College Park}
\affil[3]{School of Engineering and Applied Sciences, Harvard University}
\vspace{-1cm}
\maketitle 
\vspace{-1.5cm}
\blfootnote{\hspace{-0.6cm}2010 Mathematics Subject Classification: 91A22, 93B05, 70H25, 93C15, 92D25. Email: $^*$vidya@g.harvard.edu, $^\dagger$krishna@isr.umd.edu}

\begin{abstract}
For over fifty years, the dynamical systems perspective has had a prominent role in evolutionary biology and economics, through the lens of game theory. In particular, the study of replicator differential equations on the standard (probability) simplex, specified by fitness maps or payoff functions, has yielded insights into the temporal behavior of such systems. However behavior is influenced by context and environmental factors with a game-changing quality (i.e., fitness maps are manipulated). This paper develops a principled geometric approach to model and understand such influences by incorporating replicator dynamics into a broader control-theoretic framework. Central to our approach is the construction of a Lie algebra structure on the space of fitness maps, mapping homomorphically to the Lie algebra of replicator vector fields. This is akin to classical mechanics, where the Poisson bracket Lie algebra of functions maps to associated Hamiltonian vector fields. We show, extending the work of Svirezhev in 1972, that a trajectory of a replicator vector field is the base integral curve of a solution to a Hamiltonian system defined on the cotangent bundle of the simplex. Further, we exploit the Lie algebraic structure of fitness maps to determine controllability properties of a class of replicator systems.
\end{abstract}
\section{Introduction}
Evolutionary game theory has laid the groundwork for modeling the effect of individual interactions between members of a population comprising different phenotypes on their respective population fractions, with applications spanning several domains. Maynard Smith and Price\cite{1} illustrated the implications of small, non-fatal conflicts between conspecifics or members of competing species using the hawk-dove game example under a specific payoff (fitness) structure. In 1978, Taylor and Jonker\cite{2} offered a mathematical model for the discrete, generational update of the population frequencies on the probability simplex, as in the Maynard Smith-Price formulation. The deterministic ordinary differential equation (o.d.e) limit of the discrete frequency update is given by the replicator dynamics (vector field) with fitness linearly dependent on the population frequencies\cite{3,4}. While originating from evolutionary biology, replicator dynamics has found many applications such as modeling evolution of genotypes in population genetics\cite{5}, capturing the evolution of the frequencies of interacting species described by Lotka-Volterra equations\cite{6}, demonstrating the prevalence of certain pursuit strategies in nature\cite{7}, economics\cite{8} and more recently, in modeling the concentrations of reacting chemical species\cite{9}, to list a few. In these examples, types are interpreted variously to be genotypes, types of animal species, dyadic pursuit strategies, investment strategies, or chemical species. Here and in many other examples in the literature, we observe that the fitness of the types is either frequency independent or linear in the frequencies, and \textit{time-independent}. But system behavior is influenced by context and environmental factors suggesting that fitness maps be temporally modulated. This embeds the replicator dynamics into a broader concept of \textit{replicator control systems} where the control signals specify modulation of fitness maps and associated vector fields. In the field of geometric control theory, questions such as controllability are addressed using Lie algebras generated by the modulated vector fields\cite{10}. Here, controllability is concerned with the existence of fitness modulation that enables transfer of state on the simplex from an arbitrary initial condition to a final condition. We find that a Lie algebraic structure in the space of fitness maps (analogous to the Poisson bracket structure of classical mechanics) is effective in settling controllability properties of replicator control systems. Separately we also establish links to other aspects of classical mechanics such as gradient flows, and variational principles - in particular extending the work of Svirezhev \cite{11}. In what follows we discuss the organization and contributions of this paper.

In this work, we first focus on the algebraic and geometric characteristics of replicator dynamics defined by general fitness maps. In sections \ref{fitness} and \ref{lastruct}, we define and use a Lie algebraic structure on the space of smooth fitness maps under a \textit{replicator bracket}. Further, in section \ref{lastruct1}, we show that replicator vector fields are closed under Jacobi-Lie bracket and the passage from fitness maps to vector fields is a Lie algebra homomorphism. This property has an analogy in mechanics, where the assignment from the Poisson bracket algebra of Hamiltonian functions to the Lie algebra of Hamiltonian vector fields is a Lie algebra anti-homomorphism (figure \ref{homodiag}). We note that every simplex-preserving dynamics is a replicator dynamics of appropriate fitness when restricted to the interior of the simplex. See also Smale \cite{45}.

Fisher's fundamental theorem of natural selection states that along solutions to replicator dynamics of genetics (selection equations), the rate of change of mean fitness of a population is proportional to the additive fitness variance. This follows from the gradient property of selection equations with respect to the natural Riemannian metric of the simplex, namely the Fisher-Rao-Shahshahani metric\cite{12}. Fisher's theorem is not true for a general fitness (nonlinear in the frequencies). Kimura's maximum principle \cite{12} states that the rate of increase in mean fitness is highest along solutions to the selection equations, compared to other simplex-preserving dynamics. In contrast to these well-known optimality principles, Svirezhev \cite{11} showed that for selection and migration equations, an integral variational principle holds, akin to Hamilton's principle of least action. In section \ref{var}, we note that solutions to replicator dynamics are also extremals for a variational principle (extending Svirezhev\cite{11}). The proof of this result appears in \cite{13} and is included in the electronic \textcolor{blue}{supplementary calculations section (a)}. We then consider the existence of periodic solutions to associated Hamiltonian systems by exploiting symmetry as in Birkhoff's theorem\cite{14} on reversible systems. We illustrate these results on the one dimensional simplex using the example of Prisoner's Dilemma. 

In section \ref{cs}, we investigate replicator control systems specified by linear combinations of fitness maps modulated by control signals. We determine sufficient conditions for a class of such systems to be controllable by considering Jacobi-Lie brackets of constituent vector fields. Due to the homomorphism property noted in section \ref{lastruct1}, such conditions on the vector fields are implied by conditions on Lie algebras generated by fitness maps. It is then natural to consider optimal control problems going beyond the mechanical principles of section \ref{var} inspired by Svirezhev. In particular, using the maximum principle of Pontryagin and co-workers\cite{15} and associated Hamiltonian systems, the paper\cite{16} discusses a class of optimal control problems concerned with assimilating trajectory data from studies of starling flocks.

\textbf{Historical Remark:} As noted earlier in this introduction, the origins of replicator dynamics can be traced to multiple scientific domains \cite{3} including species competition. With respect to time-dependence of such equations, there are antecedents in the literature on predator-prey models \cite{17}. Leenheer and Aeyels \cite{18} take up the problem of controllability of equations of Lotka-Volterra type defined on the positive orthant. They also employ methods of geometric nonlinear control theory (Lie brackets of vector fields etc.) to find sufficient conditions for controllability. Connections between Lotka-Volterra equations and replicator dynamics on the simplex are well-known and discussed in \cite{6} and also in \cite{19}. But in \cite{18} these connections are not exploited, and hence they do not relate their controllability arguments to properties of fitness maps. In their papers \cite{46} and \cite{47}, building on earlier work by Hofbauer and others, Duarte and Alishah investigate Hamiltonian evolutionary games and related Poisson structures for a class of systems referred to as polymatrix games, which specialize to replicator dynamics with linear fitness under appropriate hypotheses. There does not seem to be any relation between these Poisson structures and the Lie algebra structure on the space of smooth nonlinear fitness maps given by the R-bracket in Section \ref{lastruct} of this paper.

\section{Fitness maps and replicator vector fields \label{fitness}}
Consider the $n-1$ dimensional probability simplex:
\begin{align}
\Delta^{n-1} = \left\{x = (x_1,x_2,\hdots, x_n) \in \mathbb{R}^n : 0\leq x_i \leq 1, i=1,\hdots,n, \ \sum\limits_{i=1}^{n}x_i=1 \right\}
\end{align}
Denote the interior of the simplex as follows:
\begin{align}
int(\Delta^{n-1}) = \left\{x\in \Delta^{n-1}: x_i > 0 \  \forall \ i=1,\hdots,n\right\}
\end{align}
\begin{definition}
A fitness map $f$ is a smooth map that assigns to each point in the simplex, an $n$ dimensional vector whose coordinates are smooth functions $f^i$ which denote the fitness of the $i^{th}$ type: 
\begin{align}
f: \Delta^{n-1}& \longrightarrow \mathbb{R}^n \notag
\end{align}
\begin{align}
x \ \ &\longmapsto f(x) = \left[\begin{array}{c}f^1(x) \\  f^2(x) \\  \vdots \\ f^n(x) \end{array}\right], \ f^i  \in C^\infty(\Delta^{n-1}) \ \forall \ i=1,\hdots,n
\end{align}
where $C^{\infty}(\Delta^{n-1})$ is the space of smooth real-valued functions on the simplex. 
\end{definition}
The replicator dynamics for $n$ types specifies the evolution of the probability vector $x$ associated with the propagation or diminution of the types depending on how the fitness of each type compares with the average fitness $\bar{f}(x)= \sum_i x_i f^i(x)$ in the following way:
\begin{align}
\dot{x}_i &= x_i \left(f^i(x) - \bar{f}(x) \right), \ i=1,\cdots,n \label{rep}
\end{align}
Equations (\ref{rep}) are nonlinear and simplex preserving. This can be verified by observing that $\sum_i \dot{x}_i =0$ and $x_i = 0 \implies \dot{x}_i = 0$. Due to the latter property, it is also sub-simplex preserving.  That is, any type that is extinguished remains extinguished for all future time. Letting $\Lambda(x) = \text{diag}(x_1,\hdots,x_n)$, $\textbf{e} = [1 \ \hdots \ 1]^T \in \mathbb{R}^n$, the replicator dynamics (\ref{rep}) defined by the fitness $f(x)$ can be written compactly as
\begin{align}
\dot{x} &= \Lambda(x)\left(f - \bar{f} \textbf{e}\right) \triangleq \hat{f}(x) \label{fhat}
\end{align}
where $\hat{f}(x) = [\hat{f}^1(x) \ \hat{f}^2(x) \ \hdots \ \hat{f}^n(x)]^T$ with $\hat{f}^i(x) = x_i(f^i(x) - \bar{f}(x))$. We denote by $\mathcal{X}(\Delta^{n-1})$, the set of all vector fields on the simplex, and by $\mathcal{X}_R(\Delta^{n-1})$, the family of replicator vector fields. 
\begin{remark}\label{liediffdef}
The Lie derivative $\mathcal{L}_{\hat{f}}(\phi)$ of a function $\phi \in C^{\infty}(\Delta^{n-1})$ along a replicator vector field $\hat{f}$ is denoted as
\begin{align}
X_f \phi &= \sum\limits_{i=1}^{n}\hat{f}^i\dfrac{\partial \phi}{\partial x_i} \label{liediff} 
\end{align}
\end{remark}
\begin{remark}
If $f$ is component-wise uniform, i.e., $f = \alpha(x)\mathbf{e}$ where $\alpha$ is a scalar function, then $\hat{f}=0$.  \label{compunif}
\end{remark}
\begin{remark}
If $\hat{f} = 0$ for $x\in int(\Delta^{n-1})$, $\exists \ \alpha$, a scalar function such that $f^1 = f^2 = \hdots = f^n = \alpha(x)$. \label{compunifrev}
\end{remark}
Let $T\Delta^{n-1}$ and $T_x\Delta^{n-1}$ respectively denote the tangent bundle and tangent space to the simplex at $x$. The tangent space is identified with $\{v \in \mathbb{R}^n : \sum\limits_{i}v_i = 0 \}$. From (\ref{fhat}), we see that $\hat{f}(x)\in T_x(\Delta^{n-1})$. The simplex is a Riemannian manifold with boundary equipped with the Fisher-Rao-Shahshahani (FRS) metric \cite{20,21,22} given by the metric tensor $G = [g_{ij}]$, where $g_{ij} = \delta_{ij}\dfrac{1}{x_i}, \ 1\leq i,j \leq n$ with $\delta_{ij}$ the Kronecker symbol, well defined in its interior. The inner product of two tangent vectors $v, w$ evaluated at $x \in \Delta^{n-1}$ as:
\begin{align}
\langle v, w\rangle_{FRS} = \sum\limits_{i=1}^{n}\dfrac{v_i w_i}{x_i} \label{frsmetric}
\end{align}

Although equations (\ref{rep}) suggest a particular form for the dynamics, replicator dynamics is more general than it appears, as shown in the following theorem. A result along these lines is already known to Smale \cite{45}.
\begin{theorem} \label{genrep}
Every smooth, simplex-preserving dynamics can be described by replicator dynamics with smooth fitness in the interior of the simplex.
\end{theorem}
\begin{proof}
Consider the dynamics on the simplex given by:
\begin{align}
\dot{x}_i &= \Phi^i(x), \ \ i=1,\hdots,n \label{vec}
\end{align}
where $x \in \Delta^{n-1}$. Since this dynamics preserves the simplex, then it must necessarily satisfy $\sum_i\dot{x}_i = \sum_{i}\Phi^i(x) = 0$. In $int(\Delta^{n-1})$, this condition can be used to write equations (\ref{vec}) equivalently as:
\begin{align}
\dot{x}_i &= x_i \left(\dfrac{\Phi^i(x)}{x_i}\right) \notag\\
&= x_i \left(\dfrac{\Phi^i(x)}{x_i} - \sum\limits_{j=1}^{n}x_j\dfrac{\Phi^j(x)}{x_j} \right) \notag\\
&=  x_i (f_{\Phi}^i  - \bar{f}_{\Phi}) \label{vecrep}
\end{align}
where $f_\Phi^i(x) = \dfrac{\Phi^i}{x_i}$ and $\bar{f}_\Phi = \sum\limits_{j=1}^{n}x_j f_\Phi^j = \sum\limits_{j=1}^{n}x_j \dfrac{\Phi^j}{x_j} = 0$, which is a replicator dynamics associated with smooth fitness $f_\Phi$.
\end{proof}
The viewpoint that every dynamics on the simplex is a replicator dynamics is a generalization similar to associating every dynamics on the simplex to a respective Price equation\cite{23,24}. A consequence of theorem \ref{genrep} is that one can extend the analysis presented in this work for control systems defined by any simplex-preserving dynamics. This includes the generator equation associated with a continuous time finite state Markov process, Arrhenius equations in statistical physics \cite{25}, models for infectious diseases \cite{26}, selection-recombination equations \cite{27} and selection-mutation equations\cite{28}. 

As an example, we consider a dynamics from statistical mechanics \cite{25}. Suppose that a molecule can take one of $n$ conformations. Let $x_i(t)$ be the probability that it is in conformation $i$ for $i=1,\hdots,n$. Their evolution is given by the linear, time-invariant generator equations:
\begin{align}
\dot{x} = \mathcal{R}x \label{arrhenius}
\end{align}
The matrix $\mathcal{R}$ consists of non-negative off-diagonal elements $\mathcal{R}_{ij}$ which denote the rates of probability of transition from state $i$ to $j$ when $i\neq j$, while $\mathcal{R}_{jj} = -\sum\limits_{i: i\neq j}\mathcal{R}_{ij}$. When the state transitions are produced due to thermally activated transitions between potential wells with depths $E_i$ separated by energy barriers $B_{ij}$, a special case of (\ref{arrhenius}) called Arrhenius equations is obtained, with $\mathcal{R}_{ij} = k \exp^{-\beta\left(B_{ij} - E_j\right)}$. For these equations, one can find the fitness map as $f(x) = \Lambda^{-1}(x)\mathcal{R}x$. 

The first step towards understanding the controllability properties of a nonlinear dynamical system is to consider the Lie algebra generated by the control vector fields. 
\begin{theorem}Replicator vector fields are closed under the Jacobi-Lie bracket. \label{closuretheorem}
\end{theorem}
\begin{proof}
Suppose $f_1$ and $f_2$ are two fitness maps. The Jacobi-Lie bracket of the associated replicator vector fields is
\begin{align}
[\hat{f}_1,\hat{f}_2] = \dfrac{\partial \hat{f}_2}{\partial x}\hat{f}_1 - \dfrac{\partial \hat{f}_1}{\partial x}\hat{f}_2. \notag
\end{align}
Consider the $i^{th}$ component of the bracket:
\begin{align}
[\hat{f}_1, \hat{f}_2]^i = \sum\limits_{r=1}^{n}\dfrac{\partial \hat{f}_2^i}{\partial x_r}\hat{f}_1^r - \sum\limits_{r=1}^{n}\dfrac{\partial \hat{f}_1^i}{\partial x_r}\hat{f}_2^r \notag
\end{align}
The partial derivatives in the right hand side above can be evaluated for $k=1,2$ as:
\begin{align}
\dfrac{\partial \hat{f}_k^i}{\partial x_r} &= \delta_{ir}\left(f_k^i(x) - \bar{f}_k(x)\right) + x_i\dfrac{\partial}{\partial x_r}\left(f_k^i(x) - \bar{f}_k(x)\right) \ \text{with} \notag\\
\dfrac{\partial}{\partial x_r}\left(f_k^i(x) - \bar{f}_k(x)\right) &= \dfrac{\partial f_k^i(x)}{\partial x_r} - f_k^r(x) - \sum\limits_{l=1}^{n}x_l \dfrac{\partial f_k^l(x)}{\partial x_r} . \notag
\end{align}
Substituting the above in the Lie bracket calculation,
\begin{align}
[\hat{f}_1,\hat{f}_2]^i &= \sum\limits_{r=1}^{n} \left  \{ \left[ \delta_{ir}\left(f_2^i(x) - \bar{f}_2(x)\right) + x_i \left(\dfrac{\partial f_2^i(x)}{\partial x_r} - f_2^r(x) - \sum\limits_{l=1}^{n}x_l \dfrac{\partial f_2^l}{\partial x_r}\right) \right] x_r  \left(f_1^r(x) - \bar{f}_1(x)\right)   \right \} \notag\\
&- \sum\limits_{r=1}^{n}\left\{\left[ \delta_{ir}\left(f_1^i(x) - \bar{f}_1(x)\right) + x_i \left(\dfrac{\partial f_1^i(x)}{\partial x_r} - f_1^r(x) - \sum\limits_{l=1}^{n}x_l \dfrac{\partial f_1^l}{\partial x_r}\right) \right]  x_r \left(f_2^r(x) - \bar{f}_2(x)\right)  \right \} \notag\\
\implies &[\hat{f}_1,\hat{f}_2]^i =\sum\limits_{r=1}^{n}\left \{ \delta_{ir}\left(f_2^i(x) - \bar{f}_2(x)\right) x_r \left(f_1^r(x) - \bar{f}_1(x)\right) \right \} \notag\\
&+ x_i \sum\limits_{r=1}^{n}\left \{  \left(\dfrac{\partial f_2^i(x)}{\partial x_r} - f_2^r(x) - \sum\limits_{l=1}^{n}x_l \dfrac{\partial f_2^l(x)}{\partial x_r}\right)  x_r  \left(f_1^r(x) - \bar{f}_1(x)\right)   \right \} \notag\\
&- \sum\limits_{r=1}^{n}\left \{ \delta_{ir}\left(f_1^i(x) - \bar{f}_1(x)\right)  x_r  \left(f_2^r(x) - \bar{f}_2(x)\right) \right \} \notag
\end{align}
\begin{align}
&- x_i \sum\limits_{r=1}^{n}\left \{  \left(\dfrac{\partial f_1^i(x)}{\partial x_r} - f_1^r(x) - \sum\limits_{l=1}^{n}x_l \dfrac{\partial f_1^l(x)}
{\partial x_r}\right)  x_r \left(f_2^r(x) - \bar{f}_2(x)\right)   \right \} \notag\\
\implies &[\hat{f}_1,\hat{f}_2]^i =\sum\limits_{r=1}^{n}\left \{ \delta_{ir}\left(f_2^i(x) - \bar{f}_2(x)\right) x_r \left(f_1^r(x) - \bar{f}_1(x)\right) \right \} \notag\\
&+ x_i \sum\limits_{r=1}^{n}\left \{  \left(\dfrac{\partial f_2^i(x)}{\partial x_r} - f_2^r(x) - \sum\limits_{l=1}^{n}x_l \dfrac{\partial f_2^l(x)}{\partial x_r}\right)  x_r  \left(f_1^r(x) - \bar{f}_1(x)\right)   \right \} \notag\\
&- \sum\limits_{r=1}^{n}\left \{ \delta_{ir}\left(f_1^i(x) - \bar{f}_1(x)\right)  x_r  \left(f_2^r(x) - \bar{f}_2(x)\right) \right \} \notag\\
&- x_i \sum\limits_{r=1}^{n}\left \{  \left(\dfrac{\partial f_1^i(x)}{\partial x_r} - f_1^r(x) - \sum\limits_{l=1}^{n}x_l \dfrac{\partial f_1^l(x)}{\partial x_r}\right)  x_r \left(f_2^r(x) - \bar{f}_2(x)\right)   \right \} \notag
\end{align}

Note that the first and third summations in the expression above cancel out since
\begin{align}
&\sum\limits_{r=1}^{n}\left \{ \delta_{ir}\left(f_2^i(x) - \bar{f}_2(x)\right)  x_r \left(f_1^r(x) - \bar{f}_1(x)\right) \right \} - \sum\limits_{r=1}^{n}\left \{ \delta_{ir}\left(f_1^i(x) - \bar{f}_1(x)\right) x_r  \left(f_2^r(x) - \bar{f}_2(x)\right) \right \} \notag\\
&= \left[\left(f_2^i(x) - \bar{f}_2(x)\right)x_i\left(f_1^i(x) - \bar{f}_1(x)\right)\right] - \left[\left(f_1^i(x) - \bar{f}_1(x)\right)x_i\left(f_2^i(x) - \bar{f}_2(x)\right)\right] =0 . \notag
\end{align}
Using this observation to simplify the calculation, we get:
\begin{align}
[\hat{f}_1,\hat{f}_2]^i &=  x_i \sum\limits_{r=1}^{n}\left \{  \left(\dfrac{\partial f_2^i(x)}{\partial x_r} - f_2^r(x) - \sum\limits_{l=1}^{n}x_l \dfrac{\partial f_2^l(x)}{\partial x_r}\right)  x_r  \left(f_1^r(x) - \bar{f}_1(x)\right)   \right \} \notag\\
&- x_i \sum\limits_{r=1}^{n}\left \{  \left(\dfrac{\partial f_1^i(x)}{\partial x_r} - f_1^r(x) - \sum\limits_{l=1}^{n}x_l \dfrac{\partial f_1^l(x)}{\partial x_r}\right)  x_r \left(f_2^r(x) - \bar{f}_2(x)\right)   \right \} \notag
\end{align}
Taking cue from the terms that contain summation over $l$ which appear to be like averages, we rearrange the RHS as follows:
\begin{align}
[\hat{f}_1,\hat{f}_2]^i &=  x_i \left[ \sum\limits_{r=1}^{n}x_r\left \{\dfrac{\partial f_2^i(x)}{\partial x_r}\left(f_1^r(x) - \bar{f}_1(x)\right)  -  \dfrac{\partial f_1^i(x)}{\partial x_r}  \left(f_2^r(x) - \bar{f}_2(x)\right)   \right \}\right] \notag\\
&- x_i  \left[ \sum\limits_{r=1}^{n}x_r\left \{ \sum\limits_{l=1}^{n} x_l \left( \dfrac{\partial f_2^l(x)}{\partial x_r}\left(f_1^r(x) - \bar{f}_1(x)\right)  -  \dfrac{\partial f_1^l(x)}{\partial x_r}  \left(f_2^r(x) - \bar{f}_2(x)\right) \right)  \right \} \right] \notag\\
&- x_i \left[ \sum\limits_{r=1}^{n}\left\{ f_2^r(x)x_r\left( f_1^r(x) - \bar{f}_1(x)\right)  -  f_1^r(x)x_r  \left(f_2^r(x) - \bar{f}_2(x)\right) \right \} \right]  \notag
\end{align}
The last term in the above expression vanishes since
\begin{align}
& \sum\limits_{r=1}^{n}\left \{f_2^r(x)x_r\left(f_1^r(x) - \bar{f}_1(x)\right)  -  f_1^r(x)x_r  \left(f_2^r(x) - \bar{f}_2(x)\right) \right \} \notag\\
&= \sum\limits_{r=1}^{n}x_r\left[f_2^r(x)f_1^r(x) - f_2^r(x)f_1^r(x) \right] - \sum\limits_{r=1}^{n}\left[x_r f_2^r(x)\bar{f}_1(x) - x_r\bar{f}_1(x)f_1^r(x) \right] \notag\\
&=0 - \bar{f}_1(x)\left( \sum\limits_{r=1}^{n} x_r f_2^r(x)\right)  + \bar{f}_2(x)\left(\sum\limits_{r=1}^{n}x_r f_1^r(x) \right) \notag\\
&= -\bar{f}_1(x)\bar{f}_2(x) + \bar{f}_2(x)\bar{f}_1(x) = 0. \notag
\end{align}
After exchanging the order of summation in the double summation over $r$ and $l$ in the remaining terms of $[\hat{f}_1,\hat{f}_2]^i$, we get:
\begin{align}
[\hat{f}_1,\hat{f}_2]^i &= x_i \left[ \sum\limits_{r=1}^{n}x_r\left \{\dfrac{\partial f_2^i(x)}{\partial x_r}\left(f_1^r(x) - \bar{f}_1(x)\right)  -  \dfrac{\partial f_1^i(x)}{\partial x_r}  \left(f_2^r(x) - \bar{f}_2(x)\right)   \right \}\right] \notag\\
&- x_i  \left[ \sum\limits_{l=1}^{n}x_l\left \{ \sum\limits_{r=1}^{n} x_r \left( \dfrac{\partial f_2^l(x)}{\partial x_r}\left(f_1^r(x) - \bar{f}_1(x)\right)  -  \dfrac{\partial f_1^l(x)}{\partial x_r}  \left(f_2^r(x) - \bar{f}_2(x)\right) \right)  \right \} \right] \notag\\
&= x_i\left(f_{\left\{1,2\right\}}^i(x) - \bar{f}_{\left\{1,2\right\}}(x)\right) = \hat{f}_{\left\{1,2\right\}}^i \label{lb}
\end{align}
where $f_{\left\{1,2\right\}}^i(x) = \sum\limits_{r=1}^{n}x_r\left \{\dfrac{\partial f_2^i(x)}{\partial x_r}\left(f_1^r(x) - \bar{f}_1(x)\right)  -  \dfrac{\partial f_1^i(x)}{\partial x_r}  \left(f_2^r(x) - \bar{f}_2(x)\right)   \right \}$. 
Thus we see that the Jacobi-Lie bracket $\left[ \hat{f}_1,\hat{f}_2 \right] = \hat{f}_{\left\{1,2\right\}}$ and hence by (\ref{liediff}), the commutation of the Lie differentiation operators satisfies
\begin{align}
\left[X_{f_1},X_{f_2}\right] = X_{f_{\left\{1,2\right\}}} \label{liediffrel}
\end{align}
This suggests a bracket on fitness maps and precisely such a bracket is defined in section \ref{lastruct}. 
\end{proof}
Equation (\ref{lb}) highlights an interesting property: the family of replicator vector fields is closed under the Jacobi-Lie bracket, with the fitness $f_{\left\{1,2\right\}}$ of the resultant vector field derived from fitnesses of the constituent vector fields. In general, families of vector fields are not closed under bracketing. For example, it is well known that Hamiltonian vector fields and divergence-free vector fields are closed under the Jacobi-Lie bracket, while gradient vector fields are not. Theorem \ref{closuretheorem} provides a new family where such closure property holds.

\section{Lie algebra structure of fitness maps\label{lastruct}}
Let $\mathcal{A}$ denote the commutative, associative algebra of all fitness maps from $\Delta^{n-1}$ to $\mathbb{R}^n$ with the multiplication operation $(\cdot)$ defined component-wise. That is, for two fitness maps $f$ and $g$, $\left(f \cdot g \right)^i = f^i g^i \ \forall \ i$. We define a bracket operation, termed the replicator bracket, as a map that takes two elements of the fitness algebra to produce another element of $\mathcal{A}$. 
\begin{definition} 
The replicator bracket $\left \{\cdot,\cdot \right \}_R$ (or $R$ - bracket) is defined as the map
\begin{align}
\left\{\cdot,\cdot\right\}_R : \mathcal{A} \times \mathcal{A} &\longrightarrow \mathcal{A} \notag\\
\left(f,g\right) &\longrightarrow \left\{f,g\right\}_R  = \dfrac{\partial g}{\partial x}\hat{f} - \dfrac{\partial f}{\partial x}\hat{g} \notag
\end{align}
and the components of the bracket are given by:
\begin{align}
\left\{f,g\right\}^i_R &= \dfrac{\partial g^i}{\partial x}\hat{f} - \dfrac{\partial f^i}{\partial x} \hat{g} = X_f g^i - X_g f^i \label{repbracket}
\end{align}
\end{definition}
With this definition, it can be seen that in (\ref{lb}),
\begin{align}
f_{\left\{1,2\right\}} = \left\{f_1,f_2\right\}_R \label{bracketmap}
\end{align} so that the Jacobi-Lie bracket of the replicator vector fields takes the form
\begin{align}
[\hat{f}_1,\hat{f}_2] &= \rwh{\left\{f_1,f_2\right\}_R} \label{f12}
\end{align}
As an example, let $A = [a_{ij}]$ and $B = [b_{ij}]$ be two matrices defining generator equations (\ref{arrhenius}) on the simplex. Then, the $R -$ bracket of the two associated replicator fitness maps $f_A = \Lambda^{-1}(x)Ax$ and $f_B = \Lambda^{-1}(x)Bx$ is given by:
\begin{align}
\left\{f_A,f_B\right\}_R^i &= \dfrac{\partial f_B^i}{\partial x}\hat{f}_A - \dfrac{\partial f_A^i}{\partial x}\hat{f}_B \notag\\
&= \dfrac{\partial f_B^i}{\partial x} Ax - \dfrac{\partial f_A^i}{\partial x} Bx 
\end{align}
The partial derivatives can be evaluated to be:
\begin{align}
\dfrac{\partial f_A^i}{\partial x_k} =\dfrac{1}{x_i}\dfrac{\partial (Ax)^i}{\partial x_k} - \delta_{ik}\dfrac{(Ax)^i}{x_i^2}, \ \ \dfrac{\partial f_B^i}{\partial x_k} = \dfrac{1}{x_i}\dfrac{\partial (Bx)^i}{\partial x_k} - \delta_{ik}\dfrac{(Bx)^i}{x_i^2}
\end{align}
Substituting this in the bracket calculations, we get: 
\begin{align}
&\dfrac{\partial f_B^i}{\partial x} Ax - \dfrac{\partial f_A^i}{\partial x} Bx \notag\\
&= \sum\limits_{k=1}^{n}\left(\dfrac{1}{x_i}\dfrac{\partial (Bx)^i}{\partial x_k} - \delta_{ik}\dfrac{(Bx)^i}{x_i^2}\right)(Ax)^k -  \sum\limits_{k=1}^{n}\left(\dfrac{1}{x_i}\dfrac{\partial (Ax)^i}{\partial x_k} - \delta_{ik}\dfrac{(Ax)^i}{x_i^2}\right)  (Bx)^k 
\end{align}
The terms containing $\delta_{ij}$ cancel and we get:
\begin{align}
\left\{f_A,f_B\right\}_R^i &= \dfrac{1}{x_i}\left( [b_{i1} \ b_{i2} \ \hdots \ b_{in}]Ax - [a_{i1} \ a_{i2} \ \hdots \ a_{in}]Bx\right) \ \ \implies \notag\\
\left\{f_A,f_B\right\}_R &= \Lambda^{-1}(x)(BA - AB)x = \Lambda^{-1}(x) [B,A]x
\end{align}
where $\Lambda^{-1}(x) = diag(\frac{1}{x_1}, \frac{1}{x_2},\hdots,\frac{1}{x_n})$. Notice that substituting back in the replicator equations, this takes the form of the generator equation obtained by taking the matrix commutator of $A$ and $B$ in the associated generator equations (\ref{arrhenius}). That is, we have shown that:
\begin{align}
\dot{x} &= \Lambda(x) (\left\{f_A,f_B\right\}_R - (x^T \left\{f_A,f_B\right\}_R )\textbf{e} ) \notag\\
&= \Lambda(x) \left(\Lambda^{-1}(x)(BA - AB)x - (x^T\Lambda^{-1}(x)(BA - AB)x) \textbf{e}\right)  \notag\\ 
&= \Lambda(x) \left(\Lambda^{-1}(x)(BA - AB)x\right) - \Lambda(x) \left((x^T\Lambda^{-1}(x)(BA - AB)x) \textbf{e}\right)  \notag\\ 
&= \Lambda(x) \left(\Lambda^{-1}(x)(BA - AB)x\right) -  \Lambda(x)\left((\textbf{e}^T(BA - AB)x) \textbf{e}\right)  \notag\\ 
&= (BA - AB)x \tag{$\because$ matrices $A$ and $B$ have vanishing column sums.}
\end{align}

\begin{theorem}
The set $\mathcal{A}$ together with the replicator bracket $\left \{\cdot,\cdot \right \}_R$ constitutes a Lie algebra with an ideal $\mathcal{I}$ given by component-wise uniform fitness maps:
\begin{align}
\mathcal{I} = \left\{f_\alpha = \alpha \textbf{\emph{e}}, \alpha \in C^{\infty}(\Delta^{n-1}) \right\}
\end{align}
where $C^{\infty}(\Delta^{n-1})$ is the space of smooth real valued functions on the probability simplex.
\end{theorem}
\begin{proof}
To prove the result, we show that the replicator bracket satisfies the three requisite axioms component-wise. 
\begin{itemize}
\item[(i). ]\textit{Linearity. }Consider $a$, $b \ \in \mathbb{R}$. Then, for fitness maps $f_k, k=1,2$ and $g$,
\begin{align}
&\left\{a f_1 + b f_2,g\right\}^i_R\left(x\right)\notag\\ 
&= X_{a f_1+ b f_2}g^i - X_g\left(a f_1 + b f_2\right)^i \notag\\
&= X_{a f_1}g^i + X_{b f_2}g^i - a X_gf_1^i - bX_gf_2^i \tag{by (\ref{liediff})} \notag\\
&= aX_{f_1}g^i + bX_{f_2}g^i - a X_gf_1^i - bX_gf_2^i \notag\\
&= a\left\{f_1,g\right\}^i_R + b\left\{f_2,g\right\}^i_R \tag{by regrouping terms} \notag
\end{align} 
\item[(ii). ]\textit{Skew-symmetry. }For fitness maps $f$ and $g$,
\begin{align}
\left\{g,f\right\}^i_R\left(x\right) &=  X_g f^i(x) - X_f g^i(x) \tag{by (\ref{repbracket})}\notag\\
&= -\left(X_f g^i(x) - X_g f^i (x) \right) \notag\\
&= -\left\{f,g\right\}^i_R\left(x\right) \notag
\end{align}
\item[(iii). ]\textit{Jacobi identity. }Let $f,g,h$ denote fitness maps. We want to show that 
\begin{align}
\left\{f,\left\{g,h\right\}_R \right\}^i_R\left(x\right) + \left\{g,\left\{h,f\right\}_R \right\}^i_R\left(x\right) + \left\{h,\left\{f,g\right\}_R \right\}^i_R\left(x\right) =0 \notag
\end{align}
Consider the first term:
\begin{align}
&\left\{f,\left\{g,h\right\}_R \right\}^i_R\left(x\right) \notag\\
&= X_f\left\{g,h\right\}_R^i - X_{\left\{g,h\right\}_R}f^i \tag{by definition} \notag\\
&= X_f\left(X_g h^i - X_h g^i\right) - \left[X_g,X_h\right]f^i \tag{by (\ref{liediffrel}), (\ref{repbracket}) and (\ref{bracketmap})}\notag\\
&= X_fX_g h^i - X_fX_h g^i - X_gX_h f^i + X_hX_g f^i \notag
\end{align}
Similarly, we get that:
\begin{align}
&\left\{g,\left\{h,f\right\}_R \right\}^i_R\left(x\right) = X_gX_h f^i - X_gX_f h^i - X_hX_f g^i + X_fX_h g^i, \notag\\
&\left\{h,\left\{f,g\right\}_R \right\}^i_R\left(x\right) = X_hX_f g^i - X_hX_g f^i - X_fX_g h^i + X_gX_f h^i, \notag
\end{align}
Adding up the three terms we get,
\begin{align}
&\left\{f,\left\{g,h\right\}_R \right\}^i_R\left(x\right) + \left\{g,\left\{h,f\right\}_R \right\}^i_R\left(x\right) + \left\{h,\left\{f,g\right\}_R \right\}^i_R\left(x\right) \notag\\
&= X_fX_g h^i - X_fX_h g^i - X_gX_h f^i + X_hX_g f^i +  X_gX_h f^i - X_gX_f h^i - X_hX_f g^i + X_fX_h g^i \notag\\
&+ X_hX_f g^i - X_hX_g f^i - X_fX_g h^i + X_gX_f h^i = 0.\notag
\end{align}
\end{itemize}
Hence, $(\mathcal{A},\left\{\cdot,\cdot\right\}_R)$ is a Lie algebra. From remarks (\ref{liediffdef}) and (\ref{compunif}), it is clear that $\mathcal{I}$ is an abelian sub-algebra of the Lie algebra $(\mathcal{A},\left\{\cdot,\cdot\right\}_R)$. Further, consider $f_\alpha = \alpha\textbf{e}, \alpha \in C^{\infty}(\Delta^{n-1})$ and $g \in \mathcal{A}$. Then,
\begin{align}
\left\{f_\alpha,g\right\}^i_R &= X_{f_\alpha}g^i - X_g f_\alpha^i \notag\\
&= -X_g \alpha \tag{because $X_{f_\alpha} = 0, \text{as} \ f_\alpha \in \mathcal{I}$} \notag\\
\implies \left\{f_\alpha,g\right\}_R &= -X_g\alpha\textbf{e} \in \mathcal{I} \label{ideal}
\end{align}
This shows that the replicator bracket of a component-wise uniform fitness with an element of the Lie algebra $\mathcal{A}$ produces another component-wise fitness map. Hence, the set of all component-wise uniform fitness maps denoted by the set $\mathcal{I}$ is an ideal of the Lie algebra. This concludes the proof. 
\end{proof}
\begin{remark}
We calculate the center $\mathcal{C}$ of $\mathcal{A}$. Since elements of $\mathcal{C}$ commute with all the elements of $\mathcal{A}$, we have for any $f \in \mathcal{A}$ and $c \in \mathcal{C}$,
\begin{align}
\left\{c,f\right\}_R = \dfrac{\partial f}{\partial x}\hat{c} - \dfrac{\partial c}{\partial x}\hat{f} = 0
\end{align}
Choosing $f = \alpha \textbf{\emph{e}}$, we require
\begin{align}
\left\{c,f\right\}_R = \left(\sum_j \dfrac{\partial \alpha}{\partial x_j}\hat{c}^j\right)\textbf{\emph{e}} = 0 \tag{by (\ref{ideal})}
\end{align}
For $k=1,\hdots,n$, $\alpha = x_k$, this implies $\hat{c}^k = 0$. Therefore, by remark (\ref{compunifrev}), elements of the center are component-wise uniform fitness maps. That is, $\mathcal{C} \subset \mathcal{I}$ and elements of $\mathcal{C}$ are given by $c =\eta(x)\textbf{\emph{e}}$, where $\eta(x) \in C^{\infty}(\Delta^{n-1})$. Further for any $f \in \mathcal{A}$, the bracket calculation reduces to
\begin{align}
\left\{c,f\right\}_R &= -\dfrac{\partial c}{\partial x}\hat{f} = - \left(\dfrac{\partial \eta(x)}{\partial x}\hat{f}\right) \textbf{\emph{e}} =  0. \implies \notag\\
\dfrac{\partial \eta(x)}{\partial x}\hat{f} &= 0 \ \forall f \in \mathcal{A} \implies  \notag\\
\langle \hat{g}, \hat{f} \rangle_{FRS} &= 0 \ \forall \ f \in \mathcal{A} \label{ortho}
\end{align}
where $g$ is the \emph{derived} fitness with components $g^i = \dfrac{\partial \eta(x)}{\partial x_i} \ \forall \ i$.  Since (\ref{ortho}) holds for all fitness maps in $\mathcal{A}$, this implies $\hat{g} = 0$ or equivalently that $g$ is a component-wise uniform fitness. For $c=\eta(x) \textbf{\emph{e}}$ this condition is:
\begin{align}
\dfrac{\partial \eta(x)}{\partial x_1} = \dfrac{\partial \eta(x)}{\partial x_2} = \hdots = \dfrac{\partial \eta(x)}{\partial x_n} \label{eta}
 \end{align}
Therefore, any component-wise uniform fitness is an element of the center if it satisfies (\ref{eta}) by remark (\ref{compunifrev}). We note that functions of the form 
\begin{align}
\eta(x) = \lambda\Psi(\sum_k x_k) + \lambda_0 \label{beta}
\end{align}
where $\Psi: \mathbb{R} \rightarrow \mathbb{R}$ is a function, $\lambda, \lambda_0 \in \mathbb{R}$ satisfy (\ref{eta}). However, on the simplex, $\Psi(\sum_k x_k) = \Psi(1)$, a constant. Therefore, the candidate solutions (\ref{beta}) produce elements of $\mathcal{C}$ which are of the form $\mu \textbf{\emph{e}}$, where $\mu \in \mathbb{R}$. 
\end{remark}
\begin{remark} The replicator bracket is not a derivation and hence, is not a Poisson bracket on the commutative algebra $\mathcal{A}$\cite{29}. This can be seen from the following calculations for fitness maps $f, g, h$:
\begin{align}
\left\{f\cdot g,h\right\}^i_R &= X_{f\cdot g}h^i - X_h {\left(f\cdot g\right)^i} \notag\\
&= \sum\limits_{r=1}^{n}x_r\left[ \dfrac{\partial h^i}{\partial x_r}\left(f^rg^r - \overline{f\cdot g} \right) - \dfrac{\partial (f^i g^i)}{\partial x_r}\left(h^r - \bar{h} \right) \right] \notag\\
&\neq f^i \left\{ g,h\right\}^i_R + g^i \left\{f,h\right\}^i_R , \ \text{in general}. \notag
\end{align}
Consequently, $\left(\mathcal{A}, \left\{\cdot,\cdot \right\}_R\right)$ is not a Poisson bracket algebra. However, there are two special cases for which such a relationship holds. When $f^i = \alpha \ \forall i = 1,\cdots,n$, we get:
\begin{align}
\left\{f\cdot g,h\right\}^i_R &= \sum\limits_{r=1}^{n}x_r\left[ \dfrac{\partial h^i}{\partial x_r}\left(f^rg^r - \overline{f\cdot g} \right) - \dfrac{\partial (f^i g^i)}{\partial x_r}\left(h^r - \bar{h} \right) \right] \notag\\
&= \alpha\sum\limits_{r=1}^{n}x_r\left[ \dfrac{\partial h^i}{\partial x_r}\left(g^r - \bar{g} \right)\right] - \sum\limits_{r=1}^{n}x_r\left[g^i\dfrac{\partial \alpha}{\partial x_r}\left(h^r - \bar{h} \right) + \alpha \dfrac{\partial g^i}{\partial x_r}\left(h^r - \bar{h} \right)  \right]  \notag\\
&= \alpha\left\{g,h\right\}^i_R - g^i\sum\limits_{r=1}^{n}x_r\left[\dfrac{\partial \alpha}{\partial x_r}\left(h^r - \bar{h} \right) \right] \notag\\
&= \alpha\left\{g,h\right\}^i_R + g^i\left\{f,h\right\}^i_R \tag{after adding $g^i X_f h^i$ to RHS, since $X_f = 0$.} \notag
\end{align}
Furthermore, when $h$ is a frequency independent fitness map, $\left\{f\cdot g,h\right\}^i_R  = f^i\left\{g,h\right\}^i_R + g^i\left\{f,h\right\}^i_R$.  Therefore, for these special cases, we have that:
\begin{align}
\left\{f\cdot g,h\right\}_R &= f\left\{g,h\right\}_R + g \left\{f,h\right\}_R
\end{align}
\end{remark}

\section{Lie algebra homomorphism from fitness maps to replicator vector fields\label{lastruct1}}

We say that a smooth manifold $P$ is a Poisson manifold \cite{30} if the commutative algebra $C^{\infty}(P)$ of smooth real-valued functions on $P$ carries a Lie algebra structure
\begin{align}
\left\{\cdot,\cdot\right\}_P : C^{\infty}(P) \times C^\infty(P) &\rightarrow  C^\infty(P) \notag\\
(\phi,\psi) &\mapsto \left\{\phi,\psi\right\}_P
\end{align}
satisfying the derivation property (Leibnitz rule):
\begin{align}
\left\{\phi\cdot \psi,h\right\}_P &= \phi\cdot\left\{\psi,h\right\}_P + \left\{\phi,h\right\}_P\cdot \psi
\end{align}
for all $\phi,\psi,h \in C^{\infty}(P)$. For each $h \in C^{\infty}(P)$, one assigns the hamiltonian vector field $X_h$ on $P$ by requiring
\begin{align}
X_h \phi = \left\{\phi,h\right\}_P \ \forall \ \phi \in C^{\infty}(P)
\end{align}
Denote this assignment $\text{Ham}: C^{\infty}(P) \rightarrow \mathcal{X}(P)$. Moreover from
\begin{align}
X_{\left\{h_1,h_2\right\}_P} &= - \left[X_{h_1},X_{h_2}\right]
\end{align}
it follows that Ham is an anti-homomorphism of Lie algebras. For example, when $P = \mathbb{R}^{2n}$ with canonical coordinates $z = (q_1,q_2,\hdots,q_n,p_1,p_2,\hdots,p_n)$, we have the canonical Poisson bracket 
\begin{align}
\left\{h_1,h_2\right\}_{canonical} &= \sum\limits_{i=1}^{n}\left(\dfrac{\partial h_1}{\partial q_i}\dfrac{\partial h_2}{\partial p_i} - \dfrac{\partial h_1}{\partial p_i} \dfrac{\partial h_2}{\partial q_i}\right)
\end{align}
and $X_h = J \nabla_z h$, where 
\begin{align}
J = \left[\begin{array}{cc}
0 & \mathbb{I} \\
-\mathbb{I} & 0
\end{array}\right]
\end{align}
for the canonical Poisson bracket, with $\mathbb{I}, 0$ respectively denoting identity and zero matrices. \\
\begin{figure}[t]
\centering
\includegraphics[clip, trim=1cm 4.5cm 1cm 4.5cm, width=\linewidth]{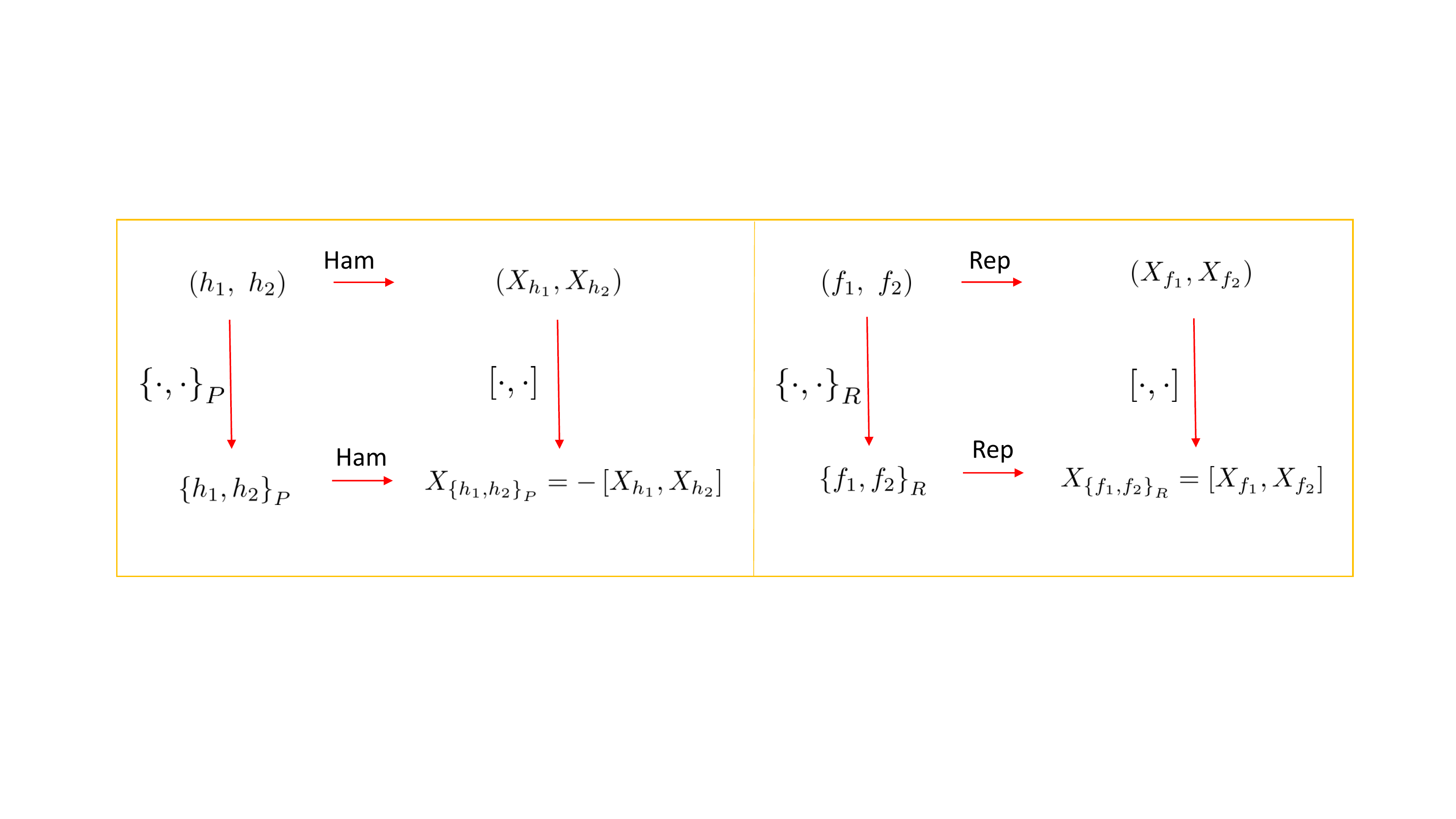}
\caption{Anti-homomorphism property between Poisson bracket algebra of Hamiltonian functions and Lie algebra of Hamiltonian vector fields is depicted in left panel, homomorphism between Lie algebra of fitness maps and Lie algebra of replicator vector fields is on the right.}\label{homodiag}
\end{figure}
We now note that there is a parallel between classical mechanics and evolutionary games. Just as Ham is an anti-homomorphism from the Poisson bracket algebra to the Lie algebra of Hamiltonian vector fields, the association of a replicator vector field to a fitness map is a homomorphism of Lie algebras. Denoting $\mathfrak{X}_R(\Delta^{n-1})$ as the Lie algebra of all replicator vector fields, we state: 
\begin{theorem}
Identifying $X_f$ with $\hat{f}$ in (\ref{fhat}) and (\ref{liediff}), the map
\begin{align}
\emph{Rep} : \mathcal{A} &\rightarrow \mathfrak{X}_R(\Delta^{n-1})\notag\\
f &\rightarrow X_f \notag
\end{align}
assigning to each fitness map a replicator vector field is a Lie algebra homomorphism with kernel given by the ideal $\mathcal{I}$ consisting of component-wise uniform fitness maps. \label{homtheorem}
\end{theorem}

\begin{proof}
Collecting together the calculations leading upto equations (\ref{liediffrel}) and (\ref{bracketmap}), we have already noted the homomorphism property in (\ref{f12}). Next, we consider the kernel of this map. Recall that elements of the ideal $\mathcal{I}$ of the Lie algebra of fitness maps $\mathcal{A}$ given by $f_{\alpha}(x)=\alpha(x)\textbf{e}$ satisfy $\hat{f}_\alpha =0$. Thus, the ideal $\mathcal{I}$ is contained in the kernel of Rep. Further, from remark (\ref{compunifrev}), the kernel of Rep is contained in ideal $\mathcal{I}$. This completes the proof. 
\end{proof}
\begin{remark}
We summarize in figure \ref{homodiag}, a comparison of \emph{Ham} and \emph{Rep}.
\end{remark}
This theorem, along with the definition of the replicator bracket, simplifies the calculation of the Jacobi-Lie bracket of two replicator vector fields. To find the Lie algebra of the vector fields generated by two replicator vector fields, one simply needs to consider the Lie algebra generated by corresponding fitness maps. The usefulness of this result is seen in section \ref{cs} on determining the controllability of a dynamics comprising replicator vector fields scaled by control inputs. Before considering such systems, we discuss an optimality principle governing general replicator dynamics in the following section.

\section{Variational principle governing simplex-preserving dynamics\label{var}}

In the previous sections, we have studied replicator dynamics on the probability simplex and noticed the Lie algebraic structure in the space of smooth fitness maps and its connections to the Lie algebra of generic vector fields $\mathcal{X}(\Delta^{n-1})$. On the other hand, the cotangent bundle $T^*(\Delta^{n-1})$ equipped with Hamiltonian functions defined on it, has a canonical Poisson bracket structure. In this section, we seek to understand how these two different structures come into contact, by considering a class of path optimization problems. Some of these ideas have been explicated in \cite{13}. 

Problems of optimization on the probability simplex have been widely known. Fisher's fundamental theorem of natural selection and Kimura's maximum principle \cite{12} provide insight into the optimizing nature of replicator dynamics. Such principles seek to understand questions of optimality in biology, in the space of strategies \cite{31}. Presently, we consider a variational problem initiated through the work of Y. M. Svirezhev \cite{11}, arising from a Lagrangian that is composed of two terms: the first term considers the norm (with respect to the Fisher-Rao-Shahshahani metric) of the velocity of a curve on the simplex and the second term is known in population genetics as the additive genetic fitness variance\cite{12}. Svirezhev showed that a cost functional given by the Lagrangian integrated over a small enough time duration is minimized by replicator dynamics with linear fitness \cite{32,33,34}. In his later work \cite{5}, it was shown that in addition to selection equations, certain gradient dynamics on the simplex such as migration and a special case of mutation equations also satisfy the Euler-Lagrange equations for this problem. Further consideration of this principle allows us to generalize this principle for any simplex preserving dynamics. Here too, replicator dynamics plays a special role.

To see this, recall that the simplex $\Delta^{n-1}$ has natural dynamics given by replicator vector fields (theorem \ref{genrep}) and also a natural geometry given by the Fisher-Rao-Shahshahani metric (\ref{frsmetric}). Putting these two together, we can consider problems on the simplex, akin to those in geometric mechanics, by defining a family of Lagrangians (Svirezhev family) using the following map:
\begin{align}
X \mapsto L_X: T\left(int\left(\Delta^{n-1}\right) \right)\rightarrow \mathbb{R}
\end{align}
where $X \in \mathcal{X}_R(\Delta^{n-1})$ is a distinguished replicator vector field and for $v_x \in T_x(\Delta^{n-1})$ the Svirezhev family of Lagrangians parametrized by $X$ is given as: 
\begin{align}
L_X(v_x) &= K(v_x,v_x) + K(X,X) = \langle v_x,v_x \rangle_{FRS} + \langle X,X \rangle_{FRS} \label{svifamily}
\end{align}
A key result about this family is that any solution to $X$ is an extremal of $L_X$. We state a theorem \cite{11,13}:

\begin{theorem} \label{varthm}
The replicator dynamics defined by the fitness $f(x) = [f^1(x) \ \hdots \ f^n(x)]^T$:
\begin{align}
\dot{x}_i &= x_i \left(f^i(x) - \bar{f}(x)\right) \ i=1,\hdots, n\label{reel}
\end{align}
satisfies the Euler-Lagrange equations associated with extremizing the following cost functional in the interior of the simplex:
\begin{align}
J = \int\limits_{t_0}^{t_1} \left[||\dot{x}||_{FRS}^2+ ||\hat{f}||^2_{FRS}\right]dt \label{cost}
\end{align}
\end{theorem}

For a proof for the setting in which the fitness is constant or linear, or derived from viewing migration or mutation equations appropriately, see \cite{5}. For the general setting of replicator dynamics, see \cite{13}. The proof is summarized in \textcolor{blue}{supplemental calculations section (a)}.

Theorem \ref{varthm} teaches us that the simplex $\Delta^{n-1}$ equipped with the FRS metric is special in that for a general manifold, the solutions to
\begin{align}
\dot{x} &= X(x)
\end{align}
are not always extremals of the Svirezhev family of Lagrangians $L_X$ defined in (\ref{svifamily}). Suppose the manifold is $Q = \mathbb{R}^n$ or an open subset of $\mathbb{R}^n$ and denote $v_x$ as $\dot{x}$. Suppose $K$ is independent of $x$ and is denoted by a constant symmetric matrix $M$ and also denote $F = X$. Then, the Euler-Lagrange equations are:
\begin{align}
M \ddot{x} &= \left(\dfrac{\partial F}{\partial x}\right)^T MF \label{elsys1}
\end{align}
On the other hand, if $\dot{x} = F(x)$, then
\begin{align}
\ddot{x} &= \left(\dfrac{\partial F}{\partial x}\right)\dot{x} = \left(\dfrac{\partial F}{\partial x}\right)F \label{elsys2}
\end{align}
For $\dot{x} = F(x)$ to be an extremal for the Svirezhev Lagrangian, we need from (\ref{elsys1}) and (\ref{elsys2}) that
\begin{align}
M \ddot{x} = \left(\dfrac{\partial F}{\partial x}\right)^T MF &= M\left(\dfrac{\partial F}{\partial x}\right)F  \notag\\
\iff  \left(M\dfrac{\partial F}{\partial x}\right)^T F &= M\left(\dfrac{\partial F}{\partial x}\right)F  \label{svicond}
\end{align}
 Condition (\ref{svicond}) is true for a constant vector field $f$, since its partial derivative is zero. Suppose $F = Ax$, (\ref{svicond}) is 
 \begin{align}
 A^TMA x &= MA^2 x \ \forall x \  \iff  \notag\\
(A^T M - MA)Ax &= 0 \ \forall x \  \iff \notag\\
\left((MA)^T - MA\right)A &= 0 \label{mcond}
 \end{align}
If $M$ is the identity matrix and $A^T = A$, $L_X$ with $X = Ax$ defines a Svirezhev family since (\ref{mcond}) is satisfied. 
\begin{corollary}
In $int(\Delta^{n-1})$, trajectories of every simplex-preserving dynamics $\dot{x}=\Phi(x)$ satisfy the Euler-Lagrange equations associated with minimizing the cost functional $J$
\begin{align}
J = \int\limits_{t_0}^{t_1}\sum\limits_{k=1}^{n} \left(\dfrac{\dot{x}_k^2}{x_k} + x_k\left(f^k - \bar{f}\right)^2\right)dt
\end{align}
where the fitness $f$ is obtained as in theorem (\ref{genrep}). \label{corollary}
\end{corollary}
Corollary (\ref{corollary}) shows that all simplex preserving dynamics are associated to variational problems.

\subsection{The Hamiltonian dynamics}
It is natural to view the variational problem stated on the tangent bundle, instead on the cotangent bundle. In terms of local coordinates for the simplex $x_i, i=1,\hdots,n-1$, the co-state or momentum variables can be defined for $i=1,\hdots,n-1$ as follows:
\begin{align}
p_i &= \dfrac{\partial L}{\partial \dot{x}_i} = \dfrac{2\dot{x}_i}{x_i} + \dfrac{2\sum\limits_{k=1}^{n-1}\dot{x}_k}{1-\sum\limits_{k=1}^{n-1}x_k}
\end{align}
or equivalently,
\begin{align}
p = 2\tilde{G}(x_1,\hdots,x_{n-1})\dot{x}
\end{align}
where $\tilde{G}=[\tilde{g}_{ij}], 1\leq i,j \leq n-1$ is the Fisher-Rao-Shahshahani metric expressed in terms of local coordinates with elements
\begin{align}
\tilde{g}_{ij} &=\delta_{ij}\dfrac{1}{x_i} + \dfrac{1}{1-\sum\limits_{k=1}^{n-1}x_k}
\end{align}
The Hamiltonian function can be written down using Legendre transform:
\begin{align}
H(x_1,\hdots,x_{n-1},p_1,\hdots,p_{n-1}) &= \dfrac{1}{2}p^T\tilde{G}^{-1}p - L(x_1,\hdots,x_{n-1},p_1,\hdots,p_{n-1}) \label{hameqn}
\end{align}
Denoting $y=[x_1 \ \hdots \ x_{n-1}]^T$ and $p=[p_1 \ \hdots \ p_{n-1}]^T$, and observing that the Lagrangian is given by
\begin{align}
L(y,p) &= \dfrac{1}{4}p^T\tilde{G}^{-1}p - V(y) \ \ \text{with} \notag\\
V(y) &= -\sum\limits_{k=1} ^{n-1}y_k (f^k(y) - \bar{f}(y))^2 - (1-\sum\limits_{k=1}^{n-1}y_k)(f^n(y)-\bar{f}(y))^2
\end{align} 
Defining $T(y,p) = \dfrac{1}{4}p^T\tilde{G}^{-1}(y)p$, (\ref{hameqn}) is equivalently,
\begin{align}
H(y,p) &= T(y,p) + V(y)
\end{align}
Therefore, the Hamiltonian function is the sum of terms $T(y,p)$ and $V(y)$ respectively resembling kinetic and potential energy terms in mechanics. From this, the Hamiltonian dynamics can be derived as:
\begin{align}
\dot{y} &= \dfrac{\partial H}{\partial p} = 	 \dfrac{1}{2}\tilde{G}^{-1}p \notag\\
\dot{p} &= -\dfrac{\partial H}{\partial y} = g(y,p) -\dfrac{\partial V}{\partial y} \label{ham}
\end{align} 
where $g(y,p)$ is the $n-1 \times 1$ vector with components $g_i(y,p) = -\dfrac{1}{4}\dfrac{\partial p^T \tilde{G}^{-1}p}{\partial y_i}$. Supplemental section \ref{supp_periodic} discusses the existence of periodic orbits for this dynamics by using Birkhoff's theorem (as in \cite{35,36}), summarizing calculations from \cite{13}.
\begin{figure}[t]
\centering
\subfigure[]{\includegraphics[clip, trim=1.5cm 6cm 1.5cm 6cm, scale=0.35]{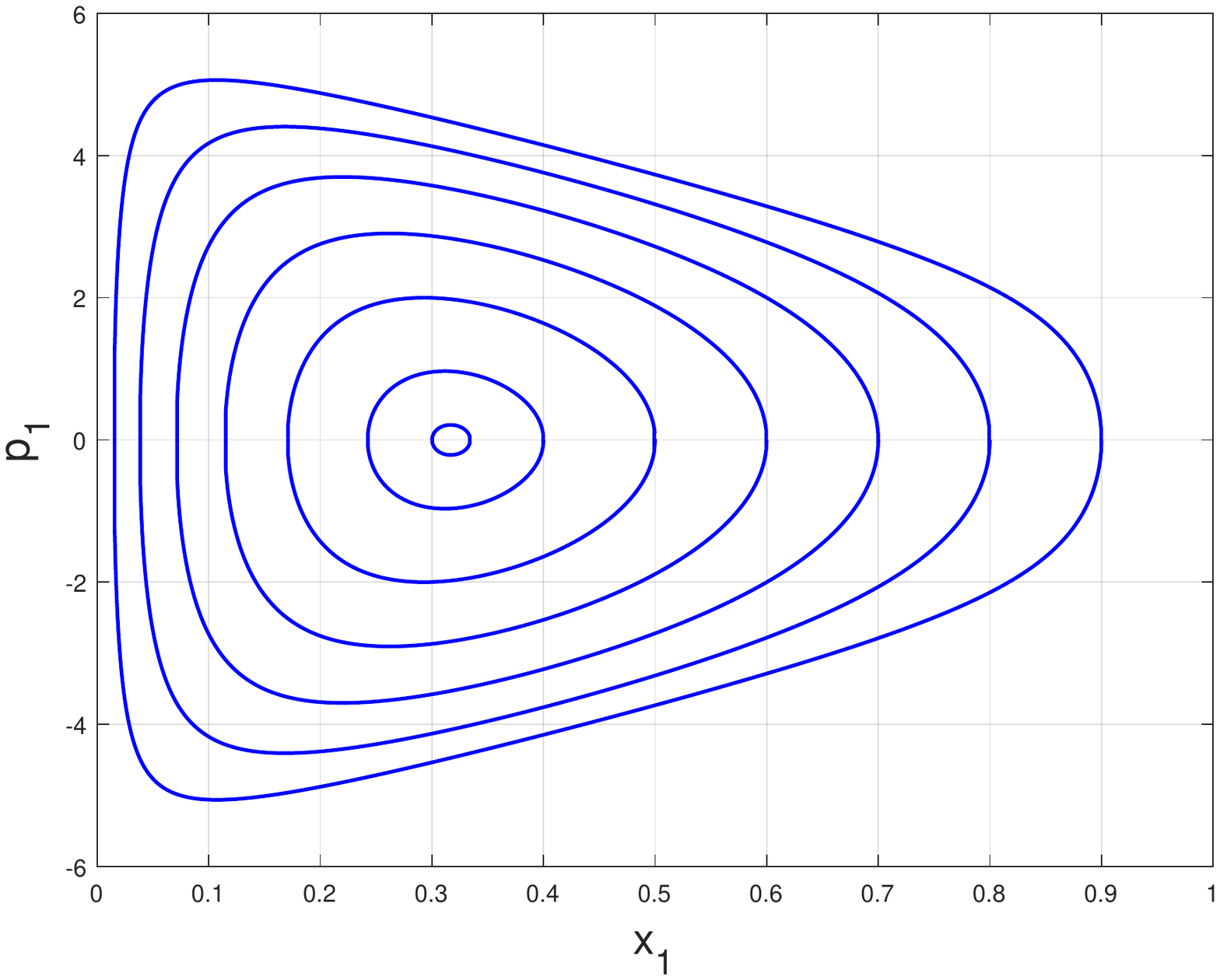}}
\subfigure[]{\includegraphics[clip, trim=1.5cm 6.5cm 1.5cm 7.5cm, width=0.45\linewidth]{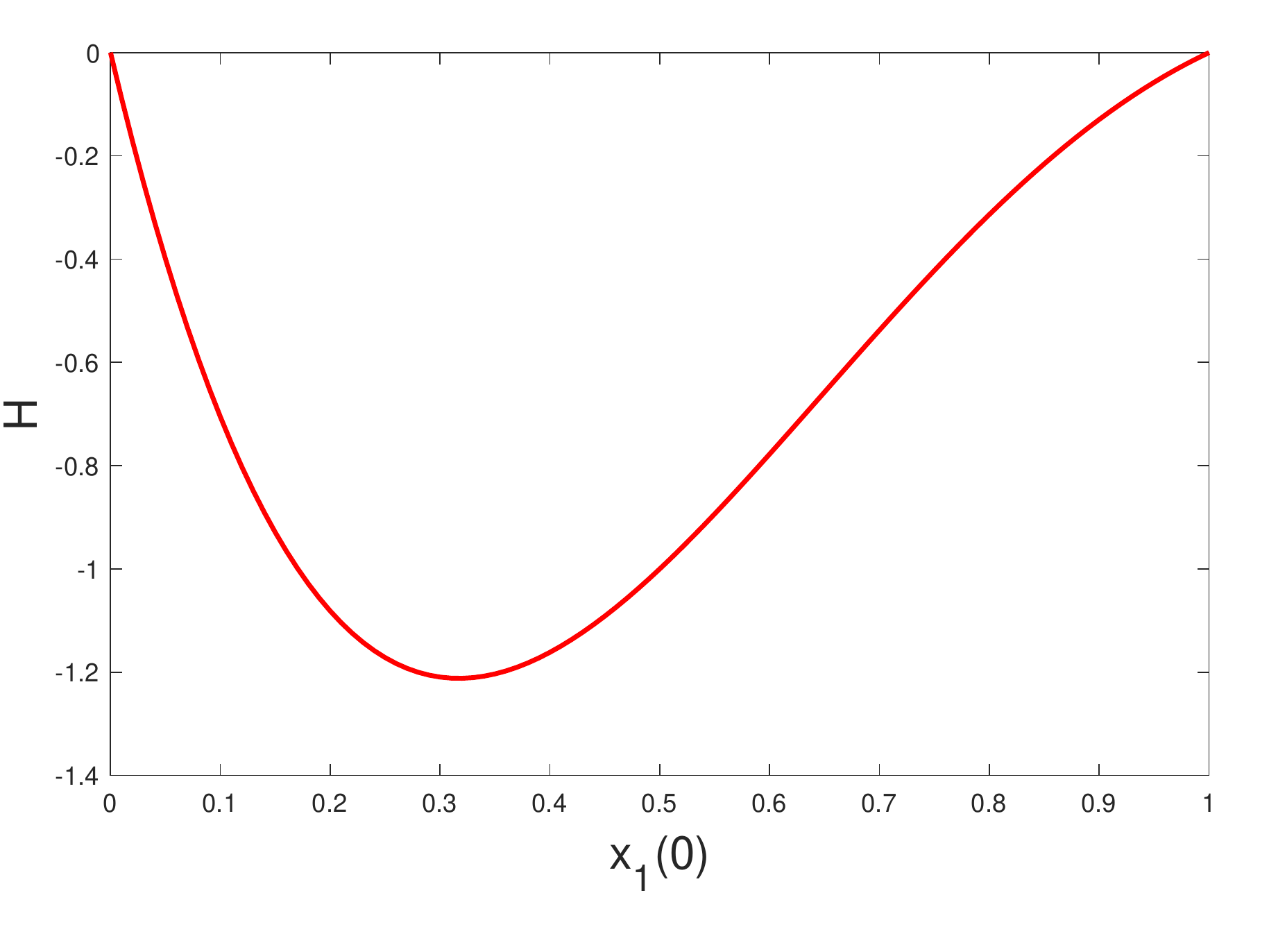}} \\
\vspace{-0.5cm}
\subfigure[]{\includegraphics[clip, trim=4cm 8cm 4cm 8cm, scale=0.45]{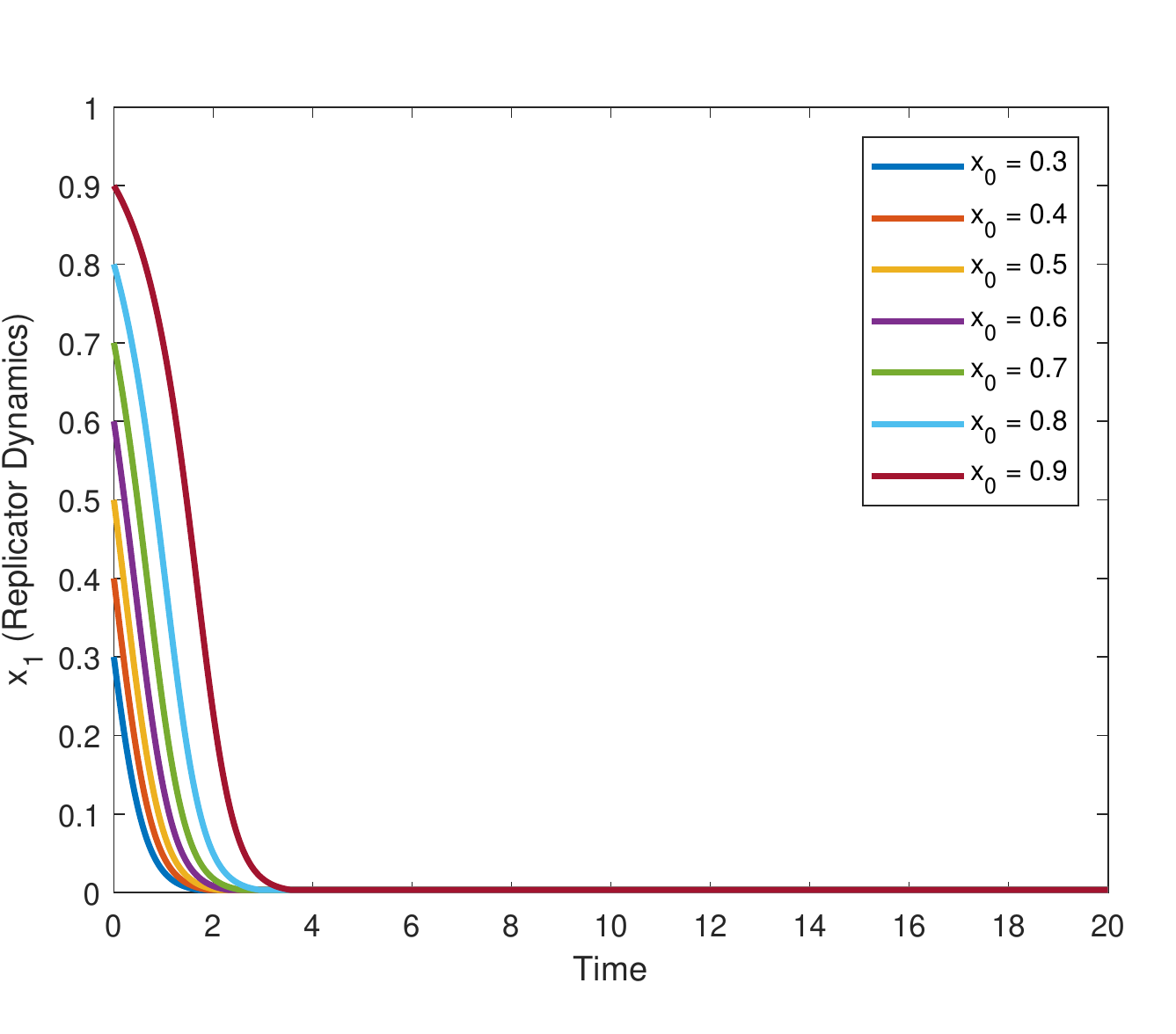}}\hspace{0.5cm}
\subfigure[]{\includegraphics[clip, trim=4cm 8cm 4cm 8cm, scale=0.45]{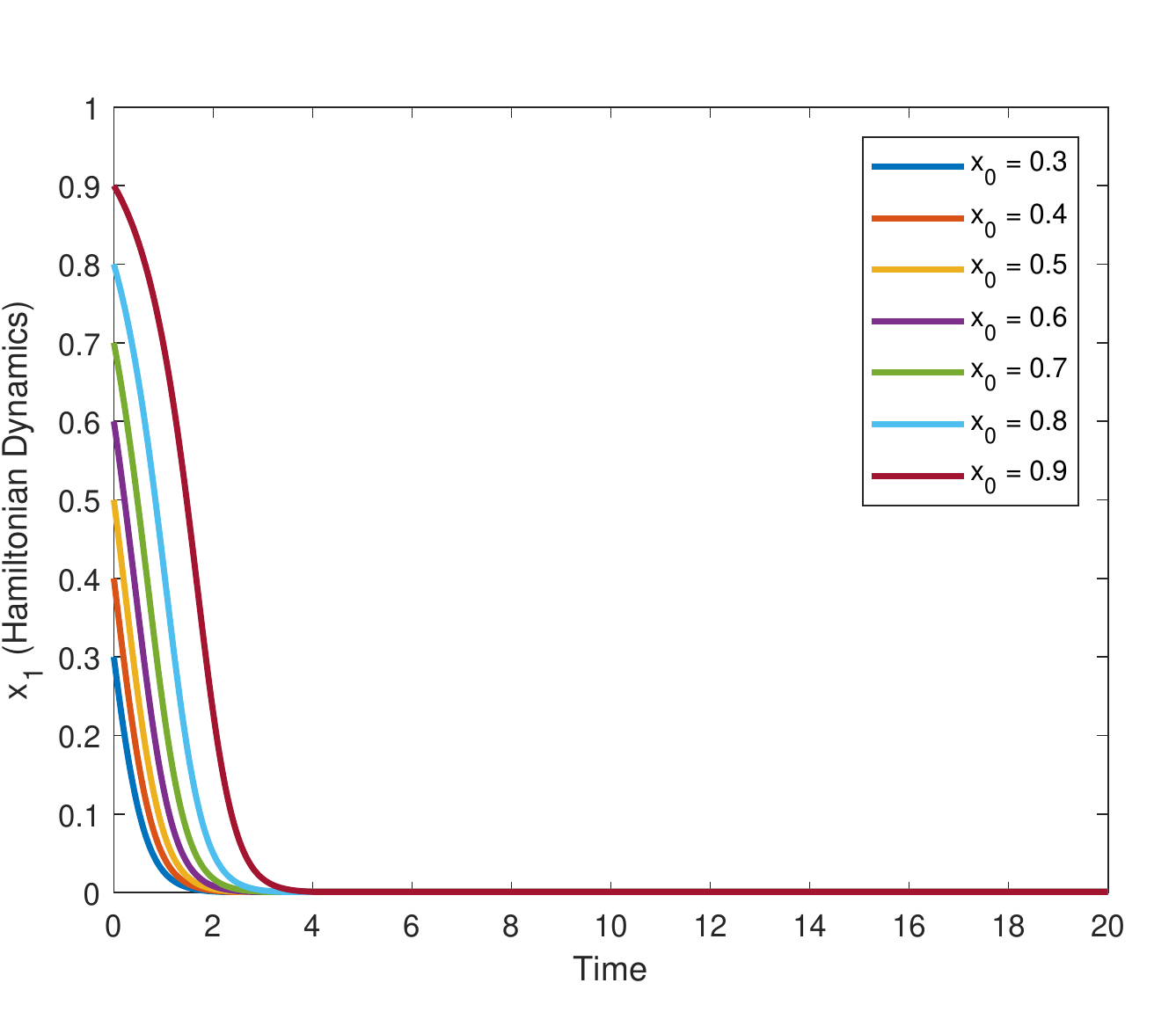}}
\caption{\textbf{Prisoner's dilemma. }(a). Simulation of the Hamiltonian dynamics for initial conditions $(x_1(0),0)$ on the $x_1 - $axis. Step size, $\Delta t = 10^{-4}$. The Hamiltonian is conserved upto an error of the order of $10^{-5}$. (b). Values of the Hamiltonian function for initial conditions on the $x_1 - $axis. (c). Simulation of replicator dynamics and (d). Hamiltonian dynamics with $x_1(0) = x_0, p_1 = 2((Ax)^1 - (Ax)^2)$ such that $H=0$.}\label{pdfig}
\end{figure}
\subsection{Hamiltonian dynamics for $n=2$}
We specialize the Hamiltonian description for dynamics on the one dimensional simplex for $n=2$, yielding a two dimensional phase space for the Hamiltonian dynamics for $f(x) = Ax$ with $A = [a_{ij}], 1\leq i,j \leq 2$. The Lagrangian, co-state and Hamiltonian are given in terms of coordinates $(x_1,p_1)$ as
\begin{align}
L(x_1,\dot{x}_1) &= \dfrac{\dot{x}_1^2}{x_1(1-x_1)} + x_1(1-x_1)\left(\left(Ax\right)^1 - \left(Ax\right)^2\right)^2, \notag\\
p_1 &= \dfrac{2\dot{x}_1}{x_1(1-x_1)}, \notag\\
H(x_1,p_1) &=  \dfrac{p_1^2 x_1(1-x_1)}{4} - x_1(1-x_1)\left(\left(Ax\right)^1 - \left(Ax\right)^2\right)^2 \label{hamfunc}
\end{align}
Suppose $a,b$ are constants such that $a = a_{11}-a_{12}-a_{21}+a_{22}$, $b = a_{12}-a_{22}$, we get:
\begin{align}
\dot{x}_1 &= \dfrac{p_1 x_1 (1-x_1)}{2}\notag\\
\dot{p}_1 &= -\dfrac{p_1^2 (1-2x_1)}{4} + (1-2x_1)(ax_1 + b)^2 + 2ax_1(1-x_1)(ax_1+b) \label{pd}
\end{align}

\subsubsection{Prisoner's dilemma}
We illustrate the solutions to the Hamiltonian dynamics through numerical simulations using the mid point rule \cite{37} for fitness given by the payoff matrix of the Prisoner's dilemma game. Here, $f = Ax$ where
\begin{align}
A = \left[\begin{array}{cc}
4 & 0 \\
5 & 3
\end{array}\right]
\end{align}
In this 2-player game, a player (the prisoner) has to choose between cooperate (strategy 1) and defect (strategy 2).  Strategy 2 dominates strategy 1 since the elements of row 2 which are the payoffs when the prisoner chooses to defect are higher, no matter the opponent's choice. However, both players receive a higher payoff if they both choose to cooperate, which is the dilemma. Since the fitness satisfies $f^2 (x)- f^1(x) = (Ax)^2-(Ax)^1 = 3-2x_1 > 0 \ \forall \ x_1 \in \Delta^{n-1}$, there are no interior equilibria for the replicator dynamics, and  the trajectories initialized with $x_1(0)\in (0,1)$ converge to $x_1 = 0$. The Hamiltonian dynamics (\ref{pd}) has an interior equilibrium point $P$ at $((3-\sqrt{3})/4,0)$. The vertices of the simplex $x_1 = 0$ and $x_1 =1$ are both invariant under the dynamics and correspond to $H=0$, as do the trajectories with initial conditions for $p_1$ corresponding to replicator dynamics. Simulation results are shown in figure \ref{pdfig}. Trajectories with initial conditions on the $x_1-$ axis correspond to periodic orbits in the phase space. Appealing to Birkhoff's theorem \cite{14}, the periodic orbits of the Hamiltonian system associated to Prisoner's Dilemma can be understood. Further, it can be seen that the value of the Hamiltonian for trajectories initialized on the $x_1-$ axis has a minimum at $P$. Figure \ref{pdfig} also depicts the identical trajectories of $x_1$ for replicator and Hamiltonian dynamics when the co-state is initialized such that $p_1 = 2( (Ax)^1 - (Ax)^2) = 2(2 x_1 - 3)$ and $H = 0$. 

It is worth noting that while the Hamiltonian system exhibits periodic behaviour, replicator dynamics on the one dimensional simplex do not. In the latter case, convergence to the vertices of the simplex or internal equilibria is typically observed. In higher dimensional simplices, periodic solutions may arise, such as in rock-paper-scissors game \cite{6} and replicator dynamics describing reaction kinetics \cite{9}.

\section{Replicator control systems\label{cs}}

We extend replicator dynamics to define a replicator control system as a driven dynamical system on $\Delta^{n-1}$:
\begin{align}
\dot{x} &= \sum\limits_{k=1}^{m}u_k \hat{f}_k \label{repcs}
\end{align}
where $\forall \ k, \ \hat{f}_k$ is the replicator vector field associated with the fitness $f_k$ and $u_k, k=1,\hdots,m$ are control inputs. The control variable has the effect of modulating the fitness. This is different from payoff modifications meant to reflect a delay in the perception of the fitness as in \cite{38}, or an addition to the fitness to attempt asymptotic stabilization about a desired equilibrium \cite{39}. Using a theorem due to Chow and Rashevskii stated below, we investigate sufficient conditions to be satisfied by the fitness maps for this system to be controllable. We first define controllability for a system defined on a manifold $M$:
\begin{align}
\dot{x} = \sum\limits_{i=1}^{m}u_iF_i(x) \label{gensys}
\end{align}
\begin{definition}
We say that the system (\ref{gensys}) is controllable if given any two points $a$ and $b$ on the smooth connected manifold $M$ wherein the system is defined, there exists a control $u(t) = (u_1(t),..,u_m(t))$, and associated solution to (\ref{gensys}) such that $x(0) = a$ and $x(T) = b$ for some $T>0$.
\end{definition}
\begin{theorem}(Chow-Rashevskii theorem, \cite{40,41,42})
The system (\ref{gensys}) defined on a connected, smooth manifold $M$ of dimension $p$, is controllable if $\forall x \in M$, there exist $p$ linearly independent vector fields in the Lie algebra generated by $\left\{F_i\right\}_{i=1}^{m}$ that span the tangent space $T_xM$ at $x$. \label{crtheorem}
\end{theorem}

\begin{remark}\label{sc}
In the following discussion, we specialize (\ref{repcs}) to replicator control systems with two constituent replicator vector fields defined by linear and frequency independent fitness respectively. 
\end{remark}

Replicator dynamics with frequency independent fitness are observed when members of a population of $n$ types are engaged in games against nature \cite{3}. In this case, trajectories initialized in the interior asymptotically approach that vertex corresponding to the type with highest fitness. When the individuals are engaged in a symmetric game against others in the population, the payoff matrix defines a linear fitness. Replicator dynamics with linear fitness exhibit rich behaviour discussed in detail in \cite{43,6}. Here, we use the homomorphism property of section \ref{lastruct1}, relating replicator bracket of fitness maps to Jacobi Lie bracket of replicator vector fields, to investigate conditions satisfied by the constituent fitness maps for controllability of this special case (\ref{repcs}). We need the following preposition as a first step in this direction.
\begin{proposition} \label{propind}
Suppose the set of fitness maps $\left\{h_i\right\}_{i=1}^{n-1}$ satisfies the point-wise property at $x\in S \subset int(\Delta^{n-1})$ for some $c_i$, $\mu$ that may depend on $x$:
\begin{align}
\sum\limits_{i=1}^{n-1}c_i h_i = \mu \mathbf{e} \implies c_i = 0 \ \forall i, x \in S.
\end{align}
Then, the associated replicator vector fields $\left\{\hat{h}_1,\hdots,\hat{h}_{n-1} \right\}$ span the tangent space at each $x \in S$. 
\end{proposition}
\begin{proof} Note that
\begin{align} 
\sum\limits_{i=1}^{n-1}c_i(x)\hat{h}_i (x) &=  \sum\limits_{i=1}^{n-1}c_i(x)\Lambda(x)\left(h_i(x) - x^T h_i(x) \textbf{e}\right) \notag\\
&= \Lambda(x) \sum\limits_{i=1}^{n-1} \left(c_i(x)h_i(x) - x^T c_i(x)h_i(x) \textbf{e}\right) \notag\\
&= \rwh{\sum\limits_{i=1}^{n-1}c_i h_i} \label{hatmap}
\end{align}
If (\ref{hatmap}) is zero, in the simplex interior, this implies that $\sum\limits_{i=1}^{n-1}c_i h_i = \mu(x)\textbf{e}$ by remark \ref{compunif}.                  
\end{proof}
We have shown in Proposition \ref{propind} that the condition for linear independence of the replicator vector fields is that no non-trivial linear combination of the constituent fitness maps make a component-wise uniform fitness map for each point in the interior of the simplex. This is a stronger condition than linear independence of the fitness maps. We use this result and the homomorphism property to specify conditions for the special class of replicator control systems noted in remark (\ref{sc}) to be controllable.

\begin{theorem}
Consider the replicator control system:
\begin{align}
\dot{x} &= \sum\limits_{k=1}^{2}u_k \hat{f}_k \label{sys2}
\end{align}
with the fitness maps given by $f_1 = [a_1 \ a_2 \ \hdots \ a_n], \  f_2 = Bx$ where $B$ is a non-singular $n \times n$ matrix. Without loss of generality, one can assume that all $a_i > 0$. Suppose further that the $a_i$ are all distinct and $B$ is non-singular. Then the system (\ref{sys2}) is controllable in the interior of $\Delta^{n-1}$.   \label{conthe}
\end{theorem}
\begin{proof}
To apply theorem \ref{crtheorem}, we consider the Lie algebra of vector fields generated by $\hat{f}_1$ and $\hat{f}_2$. However, due to the homomorphism property of section \ref{lastruct1}, we know that the Jacobi-Lie bracket of replicator vector fields can be written in terms of the replicator vector field associated to bracket of the fitness maps. That is,
\begin{align}
\left[\hat{f}_1,\hat{f}_2\right] = \rwh{\left\{f_1,f_2\right\}_R}
\end{align}
We exploit this property and investigate conditions under which (\ref{sys2}) is controllable in terms of specifications on the Lie algebra generated by the fitness maps. \\
Let $\Lambda(x) = diag(x_1,x_2,\hdots,x_n)$ and $A = diag(a_1,a_2,\hdots,a_n)$. Observing that $\dfrac{\partial f_1}{\partial x} =0$, the replicator bracket of the fitness maps is given by:
\begin{align}
\left\{f_1,f_2\right\}_R &= \dfrac{\partial f_2}{\partial x}\hat{f}_1
= BA x - (x^Ta)Bx
\end{align}
Letting  $\tilde{G} = \left\{g_k(x) = BA^{k-1}x, k=1,\hdots,n\right\}$, it can be seen that $\forall \ g_k \in \tilde{G},\  k = 1,\hdots,n$:
\begin{align}
\left\{f_1,g_k\right\}_R &= \dfrac{\partial BA^{k-1}x}{\partial x} \Lambda(x)(a - (x^Ta)e) \notag\\
&= BA^k x - (x^Ta)BA^{k-1}x = BA^k x - (x^Ta)g_k \implies \notag\\
\rwh{\left\{f_1,g_k\right\}_R} &= \rwh{BA^k x} - (x^T a)\hat{g}_k  \label{recrel}
\end{align}
Therefore, by the homomorphism property (theorem \ref{homtheorem}), by successively bracketing the vector field $\hat{f}_1$ with $\hat{g}_k$, for different values of $k$, we produce linear combinations of pairs of vector fields: one vector field that is linearly dependent on $\hat{g}_k$, and another vector field associated with the fitness $BA^k x$, which gives a potentially new tangent direction. Note that $\tilde{G}$ is not closed under the replicator bracket (this can be seen by considering brackets of the form $\left\{g_k,g_l\right\}_R$). Thus, in general, the set of fitness maps $f_1 \cup \tilde{G}$ generates an infinite dimensional Lie algebra under the replicator bracket. However, as we will see below, the subset of fitness maps $\left\{f_1, g_1, \hdots, g_n \right\}$, a finite dimensional vector space satisfies (\ref{recrel}). Due to Cayley-Hamilton theorem, it is sufficient to consider $g_k$, $k=1,\hdots,n$. Therefore, in what follows, we investigate conditions for fitness maps in the set $\tilde{G}$, under repeated $R-$ bracketing, to produce replicator vector fields that span the tangent space at each $x$ in the interior of the simplex. 

Let $G \in \mathbb{R}^{n \times n}$ denote the matrix 
\begin{align}
G = [Bx \ BAx \ BA^2x \ \hdots  \ BA^{n-1}x]
\end{align}
and $G_{k} \in \mathbb{R}^{n \times (n-1)}$ denote the matrix with columns comprising all but the $k^{th}$ column of $G$, for $k=1,\hdots,n$. That is,
\begin{align}
G_1 &= [BAx \ BA^2x \ \hdots \ BA^{n-1}x] \notag\\
G_k &= [Bx  \ \hdots BA^{k-2}x \ BA^{k}x \ \hdots \ BA^{n-1}x] \ , \ 1<k<n, \notag\\
G_n &= [Bx \ BAx \ \hdots \ BA^{n-2}x] 
\end{align}
Let $a = [a_1 \ a_2 \ \hdots \ a_n]^T$, $a^k = [a_1^k \ a_2^k \ \hdots \ a_n^k]^T$ for all $k=2,\hdots,n$ with $a^0 \triangleq \textbf{e}$, $V \in \mathbb{R}^{n \times n}$ denote the Vandermonde matrix 
\begin{align}
V = [\textbf{e}  \ a \ a^2 \ \hdots \ a^{n-1}] \label{Vmatrix}
\end{align}
and  let $V_i \in \mathbb{R}^{n \times (n-1)}$ denote the matrix obtained by removing the $i^{th}$ column vector $a^{i-1}$ from $V$. Since the components of $a$ are distinct and positive by assumption, $V$ and $V_k, k=1\hdots,n$ are all full rank, and the first $n-1$ rows of each $V_k$ form a basis for $\mathbb{R}^{n-1}$ (see \textcolor{blue}{supplemental calculations section (c)}) and \cite{44} for example, for a discussion). 

Letting $\underline{c}_k = [c_{k1} \ c_{k2} \ \hdots \ c_{k n-1}]^T$, recall by proposition \ref{propind} that if the following condition holds, 
\begin{align}
G_k \underline{c}_k &= \mu_k \textbf{e} \implies c_{ki} = 0 \ \forall \ i, \label{Gkcond}
\end{align}
then the columns of $G_k$ generate linearly independent replicator vector fields. Since $G_k = B \Lambda(x) V_{k}$, $\Lambda(x)$ is non-singular in the interior of the simplex and $B$ is non-singular by assumption, we write condition (\ref{Gkcond}) as:
\begin{align}
B \Lambda(x) V_k \underline{c}_k &=\mu_k \textbf{e} \implies c_{ki}=0 \ \forall i \notag\\
\text{or equivalently, }V_k\underline{c}_k &= \mu_k \Lambda^{-1}(x)B^{-1}\textbf{e} \implies \ c_{ki} = 0 \ \forall i \label{eqn}
\end{align}
Next, we seek to prove condition (\ref{eqn}) holds \textbf{for some} $\textbf{k}$. Let $\tilde{V}_k$ denote the invertible matrix composed of the first $n-1$ rows of $V_k$, as in \textcolor{blue}{supplemental calculations section (c)}. Since the rows of $\tilde{V}_k$ form a basis for $\mathbb{R}^{n-1}$, there exist invertible matrices $P_k$ which preserve the first $n-1$ rows of $V_k$ with its last row vector $r_k$ satisfying $r_kV_k = [  0 \ 0 \ \hdots \ 0]$ and can be represented as:
\begin{align}
P_k = \left[\begin{array}{cc}
I_{n-1} & \textbf{0}_{n-1 \times 1} \\
r_k & \
\end{array} \right]
\end{align} 
It is convenient to let the last element in $r_k$ equal to $-1$. Further it can be seen that $r_k a_{k-1}\neq 0$.
Multiplying the antecedent in (\ref{eqn}) by $P_k$ on the left, we get:
\begin{align}
P_k V_k \underline{c}_k &= \mu_k P_k \Lambda^{-1}(x)B^{-1}\textbf{e} \iff  \notag\\
\left[\begin{array}{c}
\tilde{V}_k \underbar{c}_k \\
0
\end{array}\right] &=  \left[\begin{array}{c}
\mu_k \left(P_k\Lambda(x)^{-1}B^{-1}\textbf{e}\right)_{1:n-1} \\
\mu_k \left(P_k\Lambda(x)^{-1}B^{-1}\textbf{e}\right)_n
\end{array} \right] = \left[\begin{array}{c}
\mu_k \left(\Lambda(x)^{-1}B^{-1}\textbf{e}\right)_{1:n-1} \\
\mu_k r_k \Lambda(x)^{-1}B^{-1}\textbf{e}
\end{array}\right] \label{Vn}
\end{align}
where the notation $(w)_{1:n-1}$ refers to the first $n-1$ elements of the vector $w$ and $(w)_n$ refers to its $n^{th}$ element. The antecedent of (\ref{eqn}) can be written as
\begin{align}
\tilde{V}_k \underbar{c}_k &= \mu_k \Lambda(x)^{-1}B^{-1}\textbf{e}, \ \text{and} \notag\\
0 &= \mu_k r_k\Lambda^{-1}(x)B^{-1}\textbf{e} \label{ante}
\end{align}
Now we wish to show that for any $x \in int(\Delta^{n-1})$, there is at least one $k \in \left\{1,2,\hdots,n\right\}$ such that $r_k\Lambda^{-1}(x)B^{-1}\textbf{e}\neq 0$. If we prove this, then the antecedent (\ref{ante}) says that \textit{for that} $k$, $\mu_k=0$, which implies that $\underline{c}_k=0$, which will conclude the proof. 

Consider the matrix $R$, independent of $x$ as follows:
\begin{align}
R &= \left[ \begin{array}{c}
r_1 \\
r_2 \\
\vdots \\
r_n
\end{array}\right] \in \mathbb{R}^{n \times n}
\end{align}

In the following calculations, we show that the null space of $R$ denoted $\mathcal{N}_R$ is trivial. Suppose that the column vector $v \in \mathbb{R}^n$ satisfies $Rv = \textbf{0}$. Since columns of $V$ form a basis for $\mathbb{R}^n$, we have a representation $v = \sum\limits_{l=1}^{n}d_l a^{l-1}$, $d_l \in \mathbb{R}$. Hence, 
\begin{align}
R v &= 0 \implies r_k v =0 \ \forall \ k=1,\hdots,n. \label{rkv}\notag\\
\text{Since} \ r_k v &=  r_k\left(\sum\limits_{l=1}^{n}d_l a^{l-1}\right) = r_k\left(\sum\limits_{l=1, l \neq k}^{n}d_l a^{l-1} + d_k a^{k-1}\right),\notag\\
\text{Therefore,} \ r_k v &= 0  \ \forall \ k \implies  d_k = 0 \ \forall \ k. \tag{Since $r_k a^{k-1}\neq 0 \ \forall k$} 
\end{align}
The last statement is because $r_k V_k =0$ by construction. Therefore, $\mathcal{N}_R = \left\{\textbf{0}\right\}$. Now suppose there is a $x \in int(\Delta^{n-1})$ such that $r_k\Lambda^{-1}(x)B^{-1}\textbf{e}=0 \ \forall \ k$. Then,
\begin{align}
r_k \Lambda^{-1}(x)B^{-1}\textbf{e}=0 \ \forall \ k \ \iff R\Lambda^{-1}(x)B^{-1}\textbf{e}&=0 \implies \Lambda^{-1}(x)B^{-1}\textbf{e}=0
\end{align}
by the fact that the kernel of $R$ is $\left\{\textbf{0}\right\}$. But this is impossible since $\textbf{e}\neq 0$. Hence, no such $x$ exists. That is, for each $x \in int(\Delta^{n-1})$, there exists $k$ such that $r_k\Lambda^{-1}(x)B^{-1}\textbf{e}\neq 0$.  
\end{proof}
Theorem \ref{conthe} provides a set of sufficient conditions for a special class of replicator control systems to be controllable. It is worth noting that these conditions are not affected by the dimensionality of the simplex. In \cite{18}, the accessibility of the closely related controlled Lotka-Volterra system is considered, with conditions specified on linear and constant logarithmic growth rates. Such dynamics in the absence of control projects to replicator dynamics with linear fitness on the simplex \cite{6}.

\section{Conclusion\label{conclusion}}
Ever since its introduction in \cite{2}, the replicator dynamics has played a major role in biological modeling as well as other fields. In this work, we have considered some previously unexplored Lie algebraic structure and variational principles governing simplex-preserving dynamics. We have noted that every vector field in the interior of the simplex can be viewed as a suitably defined replicator dynamics. Further, by proving that replicator vector fields are closed under the Jacobi-Lie bracket, we set the stage for examining questions of controllability. The notion of the replicator bracket on fitness maps results in a reduction of complexity of calculations involved in finding the Jacobi-Lie bracket of replicator vector fields. Sufficient conditions of Chow and Rashevskii's theorem for controllability of drift free control systems then boil down to conditions on fitness maps. This work initiates a geometric control theory of replicator control systems as in theorem \ref{conthe}, where the control inputs appear multiplicatively. 

The generality of replicator dynamics broadens the applicability of results presented here to other dynamics on the simplex. A natural consideration after ensuring controllability is one of optimality. One can formulate optimal control problems for replicator control systems that aim to minimize a cost $J = \int_{t_0}^{t_1}L dt$
subject to initial and terminal conditions, where the Lagrangian $L$ is chosen to reflect a desired optimality criterion. Some examples include, $L=1$ for time optimality and $L = ||\dot{x}||_{FRS}^2$ to obtain geodesics in a sub-Riemannian setting for a replicator control system. An interesting extension is to consider the optimal control problem for the cost in (\ref{cost}). \\
\\
\textbf{Funding. }This work was supported by the ARL/ARO MURI Program Grant No. W911NF-13-1-0390 through the University of California Davis (as prime), the ARL/ARO Grant No. W911NF-17-1-0156 through the Virginia Polytechnic Institute and State University(as prime), the National Science Foundation CPS:Medium Grant No. 1837589, and by Northrop Grumman Corporation. \\
\\
\textbf{Acknowledgements.} This paper is dedicated to the memory of Jack Riddle.

\enlargethispage{20pt}

\let\oldthebibliography\thebibliography
\let\endoldthebibliography\endthebibliography
\renewenvironment{thebibliography}[1]{
  \begin{oldthebibliography}{#1}
    \setlength{\itemsep}{0em}
    \setlength{\parskip}{0em}
}
{
  \end{oldthebibliography}
}

\newpage
\renewcommand\thesection{A}
\renewcommand{\thesubsection}{(\alph{subsection})}

\section{Supplemental Calculations}
The calculations in sections \ref{supp_proof} and \ref{supp_periodic} summarize results from \cite{13}. 
\subsection{Proof of theorem \ref{varthm}\label{supp_proof}}
Consider the Lagrangian $\mathcal{L}$ with Lagrange multiplier $\lambda$ for the cost functional $J$ in (\ref{cost}). 
\begin{align}
\mathcal{L} = \sum\limits_{k=1}^{n}\left(\dfrac{\dot{x}_k^2}{x_k} + x_k\left(f^k(x) - \bar{f}(x)\right)^2 + \lambda x_k\right) - \lambda
\end{align}
Then, the associated Euler-Lagrange equations are given by:
\begin{align}
\dfrac{d}{dt}\dfrac{\partial \mathcal{L}}{\partial \dot{x}_i} &= \dfrac{\partial \mathcal{L}}{\partial x_i} 
\end{align}
Calculating each term separately, we get:
\begin{align}
\dfrac{\partial \mathcal{L}}{\partial\dot{x}_i} &= 2\dfrac{\dot{x}_i}{x_i} \implies \dfrac{d}{dt}\dfrac{\partial \mathcal{L}}{\partial \dot{x}_i} = 2\dfrac{[x_i \ddot{x}_i -\dot{x}_i^2]}{x_i^2}, \notag\\
\dfrac{\partial \mathcal{L}}{\partial x_i} &= -\dfrac{\dot{x}_i^2}{x_i^2} + \left(f^i(x)-\bar{f}(x)\right)^2 + 2 \sum\limits_{k=1}^{n}x_k\left(f^k(x) - \bar{f}(x)\right)\left(\dfrac{\partial f^k(x)}{\partial x_i} - f^i(x) - \sum\limits_j x_j\dfrac{\partial f^j(x)}{\partial x_i}\right) + \lambda
\end{align}
which gives the Euler-Lagrange equations for $i=1,\hdots,n$ as follows:
\begin{align}
2\ddot{x}_i &= \dfrac{\dot{x}_i^2}{x_i} + x_i\left(f^i(x) - \bar{f}(x)\right)^2 + 2x_i\left[\sum_kx_k\left(f^k(x)-\bar{f}(x)\right)\left(\dfrac{\partial f^k(x)}{\partial x_i} - \sum_j x_j\dfrac{\partial f^j(x)}{\partial x_i}\right)\right] + x_i \lambda \label{el}
\end{align}
Differentiating the replicator equations,
\begin{align}
&\dot{x}_i = x_i\left(f^i(x) - \bar{f}(x)\right) \implies \notag\\
2\ddot{x}_i &= 2\left[x_i\left(f^i(x) - \bar{f}(x)\right)^2\right] +2\left[ x_i\sum_kx_k\left(f^k(x)-\bar{f}(x)\right)\left(\dfrac{\partial f^i(x)}{\partial x_k} - f^k(x) - \sum_j x_j\dfrac{\partial f^j(x)}{\partial x_k}\right) \right]\notag\\
&= 2\left[x_i\left(f^i(x)-\bar{f}(x)\right)^2 + x_i\sum_jx_k\left(f^k(x)-\bar{f}(x)\right)\left(\dfrac{\partial f^i(x)}{\partial x_k} - \sum_j x_j\dfrac{\partial f^j(x)}{\partial x_k}\right)\right]\notag\\
&- 2\left[x_i \sum_k x_k\left(f^k(x) - \bar{f}(x)\right)^2  \right], \ \text{since $\sum\limits_{i=1}^{n}x_i(f^i(x)-\bar{f}(x))=0.$} \label{xdd}
\end{align}
Comparing (\ref{el}) and (\ref{xdd}), we see that the replicator equations (\ref{reel}) satisfy the Euler-Lagrange equations with the Lagrange multiplier $\lambda = -2\sum\limits_{k=1}^{n} x_k \left(f^k(x)-\bar{f}(x)\right)^2$. 
\subsection{Existence of periodic orbits in the phase space of the simplex\label{supp_periodic}}
In this section, we state results from \cite{13} that suggest the existence of periodic trajectories for the Hamiltonian dynamics (\ref{ham}). 
\begin{definition}
A diffeomorphism $F : M \rightarrow M$ from a manifold $M$ to itself is said to be an involution if $F \neq id_M$, the identity diffeomorphism, and $F^2 = id_M$, i.e. $F(F(m)) = m,\forall \ m \in M$.
\end{definition}

\begin{definition}
A vector field $X$ defined over a manifold $M$ is said to be $F-$reversible, if there exists an involution $F$ such that: $F_*(X) = −X$; i.e. $F$ maps orbits of $X$ to orbits of $X$, reversing the time parametrization. Here $(F_{*}(X))(m) = (DF)_{F^{-1}(m)} X(F^{-1}(m)) \ $ $\forall \ m \in M$ is the push-forward of $X(m)$. We call $F$ the reverser of $X$.
\end{definition}

\begin{theorem}
(\textit{G. D. Birkhoff} \cite{14}). Let X be a $F-$reversible vector field on $M$ and $\Sigma_F$ the fixed-point set of the reverser $F$. If an orbit of $X$ through a point of $\Sigma_F$ intersects $\Sigma_F$ in another point, then it is periodic. 
\end{theorem}

We refer to \cite{35,36} for a proof of Birkhoff's theorem. The reverser $F$ in our problem is defined in the proposition below.

\begin{proposition}
The vector field defined by the Hamiltonian dynamics (\ref{ham}) is $F - $reversible, with the map $F$ given by $F(y,p) = \left(y,-p\right)$. \label{F}
\end{proposition}

\begin{proof}
Let the Hamiltonian vector field in (\ref{ham}) be denoted $X_{H}$. We note that $F$ is an involution since $F^2(y,p) = F(y,-p) = (y,p)$. We calculate the pushforward of $F$ as follows:
\begin{align}
&(F_*(X_H))(y,p) = (DF)_{F^{-1}(y,p)}X(F^{-1}(y,p)) \notag\\
&= \left[\begin{array}{cc}
\mathbb{I} & 0 \\
0 & -\mathbb{I}
\end{array}\right]\left[\begin{array}{c}
-\dfrac{1}{2}\tilde{G}^{-1}(y)p \\
g(y,-p) - \dfrac{\partial V(y)}{\partial y}
\end{array}\right] = -\left[\begin{array}{c}
\dfrac{1}{2}\tilde{G}^{-1}(y)p \\
g(y,p) - \dfrac{\partial V(y)}{\partial y}
\end{array}\right] \notag\\
&= -X_H(y,p)
\end{align}
where $0, \mathbb{I}$ are respectively $n-1$ dimensional zero and identity matrices. We are now ready to state the theorem on the existence of periodic orbits for the Hamiltonian dynamics (\ref{ham}) in the special case $n=2$ corresponding to the $1-$ dimensional simplex. 
\end{proof}

\begin{theorem}
Consider the Hamiltonian system defined on 
\begin{align}
M = \left\{\left(y,p\right): y  = \left[x_1 \ \hdots \ x_{n-1}\right] , x \in int(\Delta^{n-1}), p \in \mathbb{R}^{n-1}\right\}
\end{align}
and a frequency dependent fitness $f(x) \in \mathbb{R}^{n}$ such that the Hamiltonian function is given as
\begin{align}
H(y,p) = T(y,p)+V(y)
\end{align} 
where the kinetic energy term is $T(y,p) = \dfrac{1}{4}p^T\tilde{G}^{-1}(y)p$ and the potential energy term is $V(y) = -\sum\limits_{k=1}^{n-1}y_k\left(f^k(y)-\bar{f}\right)^2 - \left(1-\sum\limits_{k=1}^{n-1}y_k\right)\left(f^n(y)-\bar{f}\right)^2$ with the Hamilton's equations given by (\ref{ham}). For the case $n=2$, the level sets of $H$ are $1-$ dimensional in phase space. Assuming that for a fixed $c$ the level set has one connected component, then for $c<0$, the trajectory of this dynamics is a periodic orbit if the nonlinear equation $V(y)=c$ has two distinct solutions for $y \in int(\Delta^{1})$. \label{per}
\end{theorem}
\begin{proof}
Consider the map $F : M \rightarrow M$ such that $F(y,p) = (y,-p)$. By proposition \ref{F}, the Hamiltonian vector field is reversible with $F$ as the reverser. Next, we note that the fixed point set of the map 
\begin{align}
\Sigma_F = \left\{\left(y,p\right): y  = \left[x_1 \ \hdots \ x_{n-1}\right] , x \in int(\Delta^{n-1}), p=0\right\}
\end{align} 
To find the intersections of orbits in $H(t)\equiv c \neq 0$, with $\Sigma_F$, we substitute $p=0$ in the Hamiltonian to get $V(y)=c$. In the case $n=2$, the connectivity assumption means that a level set is an orbit. If the equation $V(y)=c$ has two distinct roots in $int(\Delta^1)$, the orbits in the level sets of the Hamiltonian intersect $\Sigma_F$ twice. It follows from Birkhoff's theorem that such orbits are periodic. 
\end{proof}

For the dynamics (\ref{pd}), using theorem \ref{per} above, for $H(0)=c$ the trajectories of this dynamics in the non-zero level sets of the Hamiltonian function consist of periodic orbits if the polynomial
\begin{align}
[-a^2]x_1^4 + [a^2 - 2ab]x_1^3 + [-b^2 + 2ab]x_1^2 + [b^2]x_1 + c =0 \label{cond}
\end{align}
has two distinct roots in $int(\Delta^1)$. Due to Birkhoff's theorem, the number of intersections of trajectories in the level set $H(t)\equiv c$ with $\Sigma_F$ can be obtained by solving for $x_1$ such that $p_1 = 0$. From (\ref{hamfunc}), setting $p_1=0$ is equivalent to (\ref{cond}). The dynamics has equilibria at $(0,\pm 2b), (1,\pm2(a+b))$. Additionally, there are as many equilibria $(x^*,0)$ as the number of solutions $x^* \in (0,1)$ to 
\begin{align}
(ax+b)[(1-2x_1)(ax_1+b)+2ax_1(1-x_1)]=0
\end{align}
The roots of equation (\ref{cond}) satisfy the criteria of theorem \ref{per} for initial conditions. 

\subsection{Non-singularity of Vandermonde minors\label{vdmdiscussion}}
Recall that $V_k, k=1,\hdots, n$ is the $n \times (n-1)$ matrix obtained by removing the $k^{th}$ column of $V$. That is, for $a = [a_1 \ \hdots \ a_n]^T$ and $a^k \triangleq [a_1^k  \ \hdots \ a_n^k]^T$, with $a^0 = \textbf{e}$, we have
\begin{align}
V_1 &= [a \ a^2 \ \hdots \ a^{n-1}] \notag\\
V_k &= [\textbf{e} \ a \ a^2 \ \hdots a^{k-2} \ a^k \ \hdots \ a^{n-1}], k=2, \hdots, n-1 \notag\\
V_n &= [\textbf{e} \ a \ \hdots \ a^{n-2}] \notag
\end{align} 
Let $\tilde{V_k}$ denote the $n-1 \times n-1$ matrix obtained by removing the last row of $V_k$. That is, for $\tilde{a} = [a_1 \ a_2 \ \hdots \ a_{n-1}]^T$, and $\tilde{a}^k \triangleq [a_1^k \ \hdots \ a_{n-1}^k]^T$ and $\tilde{\textbf{e}} = \tilde{a}^0$ we have 
\begin{align}
\tilde{V}_1 &= [\tilde{a} \ \tilde{a}^2 \ \hdots \ \tilde{a}^{n-1}] \notag\\
\tilde{V}_k &= [\tilde{\textbf{e}} \ \tilde{a} \ \tilde{a}^2 \ \hdots \tilde{a}^{k-2} \ \tilde{a}^k \ \hdots \ \tilde{a}^{n-1}], k=2, \hdots, n-1 \notag\\
\tilde{V}_n &= [\tilde{\textbf{e}} \ \tilde{a} \ \hdots \ \tilde{a}^{n-2}] \notag
\end{align}
Let $\tilde{V}$ denote the following Vandermonde matrix:
\begin{align}
\tilde{V} &= [\tilde{\textbf{e}}\ \tilde{a} \ \hdots \ \tilde{a}^{n-2}]
\end{align} 
If $a_i$'s are distinct and positive, both $V$ as in (\ref{Vmatrix}) and $\tilde{V}$ are non-singular. The absolute values of determinants of $\tilde{V}_k$ satisfy the following identity:
\begin{align}
|det(\tilde{V}_k)| = |det(\tilde{V})|s_{\lambda_k}(a_1,\hdots,a_{n-1}) \label{schur}
\end{align}
where $s_{\lambda_k}$ is a Schur polynomial \cite{44} of shape $\lambda_k = (\lambda_{1k}, \lambda_{2k},\hdots,\lambda_{(n-1)k})$ satisfying
\begin{align}
\lambda_{j1} &= 1 \ \forall j=1,\hdots,n-1 \notag\\
\lambda_{jk} &= 0 \ \forall j < n-k, \ \lambda_{jk} = 1 \ \forall j\geq n-k, k=2,\hdots,n-2 \notag\\
\lambda_{j (n-1)} &= 0 \ \forall j=1,\hdots,n-1 
\end{align}
Further the Schur polynomials $s_{\lambda_k}(a_1,\hdots,a_{n-1})$ can be written as the sum of monomials obtained from semi-standard Young tableaux of shape $\lambda_k$ defined by $a_i, i=1,\hdots, n-1$. Therefore, $s_{\lambda_k}(a_1,\hdots,a_{n-1}) > 0 \ \forall \ k$. This along with (\ref{schur}) and the assumptions on $a_i$ guarantee non-singularity of $\tilde{V}_k$.

\end{document}